\documentclass[12pt,a4paper,reqno]{amsart}
\usepackage{graphicx}
\usepackage{yhmath}
\usepackage{amsmath,verbatim,enumerate}
\usepackage[dvipsnames]{xcolor}
\usepackage{amssymb}
\usepackage{pgf,tikz}
\usetikzlibrary{arrows}
\usepackage[utf8]{inputenc}
\usepackage{subfig}
\usepackage{hyperref}

\definecolor{Dag}{RGB}{20, 100, 20}

\theoremstyle{plain}
\newtheorem{thm}{Theorem}[section]
\newtheorem{lemma}[thm]{Lemma}

\newtoks\prt
\newtheorem{proclaim}[thm]{\the\prt}
\theoremstyle{definition}

\newtheorem{remark}[thm]{Remark}

\newtheorem{definition}[thm]{Definition}

\def\eqn#1$$#2$${\begin{equation}\label#1#2\end{equation}}

\numberwithin{equation}{section}

\def\bal{\begin{aligned}}
	\def\eal{\end{aligned}}
\def\u#1{\hbox{\boldmath $#1$}}
\def\B{\mathcal{B}}
\def\G{\mathcal{G}}

\def\diam{\operatorname{diam}}
\def\dist{\operatorname{dist}}
\def\card{\operatorname{card}}
\def\loc{\operatorname{loc}}
\def\epsilon{\varepsilon}

\def\en{\mathbb N}
\def\er{\mathbb R}

\def\H{\mathcal{H}}
\def\oint{-\hskip -13pt \int}

\def\tO{\u{O}}
\def\tP{\u{P}}

\def\Q{\mathcal{Q}}

\def\r2{\er^2}

\def\sgn{\operatorname{sgn}}

\def\co{\operatorname{co}}

\def\rn{\mathbb R^n}

\def\sgn{\operatorname{sgn}}

\def\Z{\mathcal Z}
\def\N{\mathcal N}

\newcommand{\dx}{\,\textup{d}}

\definecolor{ariwarn}{RGB}{100, 0, 190}

\makeatletter
\newcommand{\labeltext}[2]{%
	\@bsphack
	\def\@currentlabel{#1}{\label{#2}}%
	\@esphack
}
\makeatother

\def\step#1#2#3{\par \noindent{{\vskip 5pt \bf Step~\labeltext{#1}{#3}#1. }{\bf #2. }}}

\newtoks\by
\newtoks\paper
\newtoks\book
\newtoks\jour
\newtoks\yr
\newtoks\pages
\newtoks\vol
\newtoks\publ
\def\ota{{\hbox\vol{???}}}
\def\cLear{\by=\ota\paper=\ota\book=\ota\jour=\ota\yr=\ota
	\pages=\ota\vol=\ota\publ=\ota}
\def\endpaper{\the\by, {\the\paper},
	\textit{\the\jour} \textbf{\the\vol} (\the\yr), \the\pages.\cLear}
\def\endbook{\the\by, \textit{\the\book}, \the\publ.\cLear}
\def\endprep{\the\by, \textit{\the\paper}, \the\jour.\cLear}
\def\endyearprep{\the\by, \textit{\the\paper}, \the\jour, (\the\yr).\cLear}
\def\name#1#2{#2 #1}
\def\nom{ \rm no. }

\title[Diffeomorphic approximation of Planar Sobolev Homeomorphisms]{Diffeomorphic approximation of Planar Sobolev Homeomorphisms in rearrangement invariant spaces}
\author[D.~Campbell]{Daniel Campbell}
\address{D.~Campbell: Department of Mathematics, University of Hradec Kr\' alov\' e, Rokitansk\'eho 62, 500 03 Hradec Kr\'alov\'e, Czech Republic}
\address{Faculty of Economics, University of South Bohemia, Studentsk\' a 13, Cesk\' e Budejovice, Czech Republic}
\email{\tt daniel.campbell@uhk.cz}
\author[L.~Greco]{Luigi Greco}
\address{L.~Greco: Dipartimento di Ingegneria Elettrica e delle Tecnologie dell'In\-for\-ma\-zio\-ne, Universit\`a degli Studi di Napoli ``Federico II", Via Claudio~21, 80125 Napoli, Italy}
\email{luigreco@unina.it}
\author[R.~Schiattarella]{Roberta Schiattarella}
\address{R.~Schiattarella: Dipartimento di Matematica e Applicazioni ``R. Caccioppoli",
	Universit\`a degli Studi di Napoli  ``Federico II",
	Via Cintia, 80126 Napoli, Italy}
\email{roberta.schiattarella@unina.it}
\author[F.~Soudsk\' y]{Filip Soudsk\' y}
\address{F. Soudsk\' y: Department of Mathematics and Didactic of Mathematics, Faculty of Science, Humanities and Education, TECHNICAL UNIVERSITY OF LIBEREC, Studentsk\' a 1402/2}
\email{\tt filip.soudsky@tul.cz}
\address{Department of Mathematics, University of Hradec Kr\' alov\' e, Rokitansk\'eho 62, 500 03 Hradec Kr\'alov\'e, Czech Republic}
\email{filip.soudsky@uhk.cz}

\thanks{The first and fourth authors were supported by the grant GACR 20-19018Y. L.G. and R.S. are members of Gruppo Nazionale per l'Analisi Matematica, la Probabilit\`a e le loro Applicazioni (GNAMPA) of INdAM. The research of R.S. has been funded by PRIN Project 2017JFFHSH}

\begin{document}
	
	\begin{abstract}
		Let $\Omega\subseteq\er^2$ be a domain, let $X$ be a rearrangement invariant space and let $f\in W^{1}X(\Omega,\er^2)$ be a homeomorphism between $\Omega$ and $f(\Omega)$. Then there exists a sequence of diffeomorphisms $f_k$ converging to $f$ in the space $W^{1}X(\Omega,\er^2)$.
	\end{abstract}

	\maketitle
	
	\section{Introduction and main results}
	Recently, motivated by applications in non--linear elasticity and in geometric function theory, a great deal has been devoted in understanding the question of approximating homeomorphisms $f\colon \Omega\subset \mathbb R^n\to  f(\Omega)\subset \mathbb R^n$ with either diffeomorphisms or piece-wise affine homeomorphisms. This problem is not trivial because the usual approximation techniques like mollification or Lipschitz extension using maximal operator destroy, in general, the injectivity.

	In variational models of nonlinear elastic deformations of solid flexible bodies we search for minimisers of energy functionals (often) of the form
	$$
	I(f)=\int_{\Omega} W(Df)\,dx\,,
	$$
	where $W:\er^{n\times n}\to\er$ is a stored-energy functional satisfying
	\eqn{E_WA}
	$$
	W(A)\to+\infty\quad\text{as $\det A\to 0$} \qquad W(A)=+\infty\quad\text{if $\det A\leq0$}\, .
	$$
	We require that our model respects the law of non-interpenetration of matter and, assuming that the body does not fracture or break, it is therefore natural to look for a minimiser among homeomorphisms. Therefore we minimise the functional over Sobolev homeomorphisms satisfying given boundary values.
	
	Intuition gives the impression that the minimising deformation should be in essence a diffeomorphism (say up to a null set). A naive perception is that a Sobolev homeomorphism is essentially a diffeomorphism. In fact the question of the regularity of minimisers and the question of the behaviour of Sobolev homeomorphisms are somewhat inter-related. A key step to proving the regularity of minimisers (see \cite{Ball, Ball2}) is to show that any Sobolev homeomorphism can be approximated arbitrarily well by diffeomorphisms. This is the so-called Ball-Evan's approximation question and is currently a topic of much interest. The initial breakthrough in the planar case were the papers \cite{IKO} and \cite{IKO2} for $f\in W^{1,p}$, $p>1$. This was followed by \cite{HP}, planar homeomorphisms in $W^{1,1}$. The latter techniques have further been developed in \cite{BiP} (bi-Sobolev $W^{1,1}$ case), \cite{PR, PR2} (BV case) and \cite{C} (Orlicz-Sobolev case). There are still many open questions in this context, especially $W^{1,p}$-bi-Sobolev and dimension $n=3$.
	
	Naturally, given that one can approximate homeomorphisms by diffeomorphisms in the Orlicz-Sobolev sense (see \cite{C}), the question of approximation in other classes of function spaces such as Lorentz Sobolev spaces, or Grand Sobolev spaces arises. These classes are an important tool in studying the regularity of solutions of certain PDEs and variational problems and usually provide sharper results for existence and regularity of a solution. To provide results for all these important classes at once we will study the question of approximation in general Banach function space. The result from \cite{C} gives us a strong indication that a similar result should hold under a more general context. On the other hand in general r.i. spaces one lacks the explicit norm expression utilised for that result and there are several obstacles that must be overcome. 
	
	Let us just recall that by a r.i. space we mean a Banach function space $X(\Omega)$ on the domain $\Omega$, endowed with a norm $\| \cdot \|_{X(\Omega)}$ such that
	$$
	\| u\|_{X(\Omega)}= \| v\|_{X(\Omega)} \;\,\text{whenever } \;\, u^*= v^*
	$$
	where $u^*$ and  $v^*$ denote the decreasing rearrangements of the functions $u$, $v$.
	The Sobolev space over $X$ is defined as
	$$
	W^1X(\Omega,\mathbb R^2)=\left\{ f\in W^{1,1}_{\text{loc}}(\Omega, \mathbb R^2): \|f\|_{W^1X(\Omega)}= \| Df\|_{X(\Omega)} + \| f\|_{X(\Omega)} <\infty\right\}.
	$$
	For the full definition see Definition~\ref{RIS}.
	
	In order to introduce our main result we include the definition of the Lebesgue point property. This property is stronger than the absolute continuity (see \eqref{EpsDel}) of the norm and has been thoroughly characterised in \cite{CCPS}.
	\begin{definition}[Lebesgue point property]\label{LebesgueProperty}
		Let $\Omega \subset \rn$ be measurable. We say that a function space $X(\Omega)$ satisfies the Lebesgue property if for all $u\in X(\Omega)$ and almost all $x\in \Omega$ one has
		\begin{equation}\label{LPPinX}
		\lim\limits_{r\to 0+}\frac{\|[u - u(x)]\chi_{B(x,r)}\|_{X(\Omega)}}{|B(x,r)|} =0.
		\end{equation}
		We refer to the points $x$ for which \eqref{LPPinX} hold as Lebesgue points of $u$ in $X$.
	\end{definition}
	By a finitely connected domain $\Omega \subset \er^2$ we refer to a domain $\Omega \subset \er^2$ such that $\er^2 \setminus \Omega$ has a finite number of components.
	\begin{thm}\label{main}
		Let $\Omega\subseteq\er^2$ be a finitely connected domain. Let $X(\Omega)$ be a rearrangement-invariant Banach function space satisfying the Lebesgue point property (see~Definition~\ref{LebesgueProperty}). Let $f\in W^1 X(\Omega, \er^2)$ be a homeomorphism. For arbitrary $\varepsilon>0$ there exists a diffeomorphism $\tilde{f}$ such that
		$$
		\|Df-D\tilde{f}\|_{X(\Omega)}<\varepsilon \quad\textup{and}\quad \|f-\tilde{f}\|_{L^\infty(\Omega)}<\varepsilon.
		$$
	\end{thm}
	
	\begin{remark}
		Not only can we approximate a Sobolev homeomorphism by diffeomorphisms but by locally finite piece-wise affine homeomorphisms in r.i. spaces satisfying the Lebesgue point property. Moreover, if $\partial \Omega$ is a polygon on which $f$ is piece-wise linear then we can approximate $f$ by finitely piece-wise affine homeomorphisms.
	\end{remark}
	
	\begin{remark}
		Theorem~\ref{main} provides us with a diffeomorphic approximation of homeomorphisms in customary classes of r.i.\ spaces. Indeed, besides recovering the case of Sobolev-Orlicz space $W^{1, \Phi}$, where $\Phi$ is a Young function satisfying $\Delta_2$-condition, we can consider also the Sobolev-Lorentz spaces, taking $X=L^{p,q}$ for $1\leq q\leq p<\infty$. Another interesting case is $X=\Lambda_\varphi$, the Lorentz endpoint space associated with a (non-identically vanishing) concave function $\varphi\colon [0,\infty)\to[0,\infty)$ satisfying $\lim_{s\to0+}\varphi(s)=0$. For details, see~\cite{CCPS}.
	\end{remark}

	It is not hard to observe that the absolute continuity of the norm of $X$ is a necessary condition for diffeomorphic approximation of homeomorphisms in $W^{1}X$ and by \cite{CCPS} it is also necessary for the Lebesgue point property. Our technique relies heavily on estimates derived directly from the Lebesgue point property. The authors hypothesise that the Lebesgue point property is in fact necessary for the approximation of $f\in W^{1}X$ by smooth functions (independent of injectivity).
	
	\subsection{A brief description of the proof of Theorem~\ref{main}}
	
	In this subsection we outline the basic plan of our proof of Theorem~\ref{main}. As suggested above the general concept is similar to that in \cite{HP} and \cite{C}. Assume that we have a homeomorphic and locally-finite piece-wise affine approximation of $f$. We can then approximate these homeomorphisms by diffeomorphisms using \cite{MP}. The diffeomorphisms from this result coincide with the original piece-wise affine homeomorphisms up to a tiny set. Further they have the same Lipschitz constant as the approximated map up to a bounded multiplicative constant. The combination of the above two facts with the absolute continuity of the norm of $X$ means that the diffeomorphisms given by \cite{MP} also converge to the piece-wise affine homeomorphisms in $W^{1}X$. The entire argument is in Lemma~\ref{Approx} and thanks to this, the question reduces to approximating by piece-wise affine homeomorphisms. 
	
	In Lemma~\ref{Segregation} we create nested subdomains $\Omega_{k}$ of $\Omega$ and a grid of squares of a given size is made in each $\Omega_{k} \setminus \overline\Omega_{ k-1}$ so that the following holds
	\begin{itemize}
		\item all the squares have the same size and fill all of $\Omega_{k} \setminus \overline\Omega_{k-1}$ except a set so small that the norm of the restriction of $f$ to this set is bounded by $2^{-k}\epsilon$,
		\item thanks to the Lebesgue point property for $Df$, $f$ is very close to an affine function except for some squares whose union has measure so small that the norm of the restriction of $|Df|$ to this set is bounded by $2^{-k}\epsilon$.
	\end{itemize}
	In general it is necessary to know that the behaviour of $f$ is reasonable on the boundary of the squares. This is not automatically true but by slightly moving the boundaries of the squares it becomes true, which is achieved in Lemma~\ref{GridLock}.
	
	Each of the squares can be split into 2 triangles by dividing along a diagonal (say the southwest-northeast diagonal). On the squares where $f$ is very close to a nice affine map (Jacobian not too small, derivative not too big or too small) the map which is affine on each of the pair of triangles with values coinciding with $f$ at the vertices of the square approximates $f$ well. If the Jacobian is zero we use the result Theorem~\ref{extension2} to approximate. Now either the map is close to a constant on the given square or is not close to any linear map. In either case we use Theorem~\ref{extension} to define our piece-wise affine approximation. In the last two cases the smallness either of $Df$ or of the set is enough to make sure that the error is small.

	Finally we have Lemma~\ref{Click} to fill the small space around the boundary of $\Omega_{ k}$. Because the size of the set is so small we get that the norm of the map and the approximation is less than $2^{- k}\epsilon$.
	
	\section{Preliminaries}
	In this section we shortly list the basic notation that will be used throughout the paper. The set $Q(c,r) = \{(x,y)\in\er^2: |x -c_1| \leq r, |y-c_2|\leq r\}$ will denote the closed square centred at $c$ with side length $2r$. Similarly, $B(c,r) = \{(x,y)\in \er^2 : |(x,y) - c| < r \}$ is the open ball centred at $c$ with radius $r$. For the ease of notation, for $t>0$ we will denote $t Q(c,r)=Q(c,tr)$, and $tB(c,r)=B(c,tr)$.
	
	Sometimes we will work on 1-dimensional objects in $\er^2$, which can be parametrised by a Lipschitz curve $\varphi\colon [0,1]\to\er^2$, for example segments, and various polygons. We may assume that our $\varphi$ is one-to-one and $|\varphi'|$ is constant almost everywhere. For almost all $t\in (0,1)$ there exists a vector $\frac{\varphi'(t)}{|\varphi'(t)|}$which we call the tangential vector at the point $\varphi(t)$ and denote this vector as $\tau = \tau(\varphi(t))$. If a mapping $f$ is defined on $\varphi([0,1])$ and $f\circ\varphi$ is absolutely continuous, then we call
	$$
	D_{\tau}f = \frac{(f\circ\varphi)'}{|\varphi'|}
	$$
	tangential derivative along the curve $\varphi$.

	We take advantage of standard denotation of average integrals using the symbol $-\hskip -10pt \int$. Since we integrate with respect to different measures, we emphasise the fact that we divide the integral by the measure of the set we integrated over, {\it where we measure the set with the same measure used in the integral}.
	
	Through out the whole paper we will assume without loss of generality (see e.g.~\cite[Theorem 5.22]{HK}) that
	\begin{equation}\label{JPAE}
	J_f\geq 0 \text{ a.e.}
	\end{equation} 
	
	\subsection{Rearrangement invariant function spaces}
	
	Here we collect all the background material that will be used in the paper.
	
	Let $E$ be a (non-negligibile) Lebesgue measurable subset of $\mathbb R^n$ of finite measure. We denote by $\mathcal{L}^n(E)$ its Lebesgue measure. We set 
	$$
	L^0(E)=\left\{ f: f \text{ is measurable function on } E \text{ with values in }[-\infty, + \infty] \right\}
	$$
	and
	$$L^0_+(E)= \left\{ f\in L^0(E): f\geq 0\right\}.$$
	From now on we shall identify functions $f_1,f_2$ for which $\mathcal{L}^n(\{f_1\neq f_2\})=0$ in the space $L^0$.
	The non-increasing rearrangement $f^*\colon [0, +\infty] \rightarrow[0, +\infty]$ of a function $f\in L^0(E)$ is defined by
	$$
	f^*(s)=\inf\left\{ t\geq 0: \mathcal{L}^n\big (\left\{x\in E: |f(x)|>t  \right\}\big) \leq s \right\} \qquad s\in [0, + \infty)
	$$
	and the \textit{Hardy-Littlewood maximal function} $f^{**}:(0,\mathcal{L}^n(E))\to [0,+\infty)$ is given by
	$$
	f^{**}(t):=\frac{1}{t}\int_0^t f^*(s)\dx s.
	$$
	
	\begin{definition}\label{RIS}
		Let $E$ be a Lebesgue measurable subset of $\mathbb R^n$ and let $\| \cdot \|_{X(E)}: L^0(E) \rightarrow [0, +\infty]$ be a functional. Consider the following properties
		\begin{enumerate}[\upshape(P1)]
	\item $\|\cdot\|_{X(E)}$ is a norm on $L^0(E)$.\label{primo}
			\item For all $f,g\in L^0_+(E)$ the inequality $f(x)\leq g(x) \textup{ for }\text{a.e. \textit{x} in } E \text{ implies }\\ \| f\|_{X(E)}\leq \|g\|_{X(E)}$. 
			\item $\sup\limits_{k}\| f_k\|_{X(E)}= \| f\|_{X(E)}$ if $0\leq f_k(x) \nearrow f(x)$ for a.e. $x$ in E.
			
			\item Let $G\subset E$ be a set of a finite measure. Then
			$$
			\|\chi_G\|_{X(E)}<\infty.
			$$
			\item Let $G\subset E$ be a set of a finite measure. Then there exists a constant $C_G$ depending only on the choice of the set $G$ for which
			$$
			\|f\chi_G\|_{L^1(E)}\leq C_G\|f\chi_G\|_{X(E)}.
			$$
			for all $f\in L^0(E)$.
			
			\item \label{ultimo}$$
			\| f\|_{X(E)}= \|g\|_{X(E)}\text{ whenever } f^*= g^* \text{\emph{(rearrangement invariance)}}
			$$
		\end{enumerate}
		If $\|\cdot\|_X$ enjoys the properties (P1)-(P5) we call it a Banach function norm. If it also enjoys (P6) we call it a rearrangement invariant Banach function norm. Let $\|\cdot\|_X$ be (rearrangement-invariant) Banach function norm then we call set
		$$
		X(E):=\{f\in L^0(E):\|f\|_{X(E)}<\infty\}
		$$
		endowed with the norm $\|\cdot\|_{X(E)}$ a (rearrangement-invariant) Banach function space.
	\end{definition}

	Following the properties (P2),(P4), (P5) one can observe that if $\mathcal{L}^n(E)<\infty$ for arbitrary Banach function space, the following holds true
	\begin{equation}\label{RISANDWICH}
	L^\infty(E) \hookrightarrow X(E) \hookrightarrow L^1(E)    
	\end{equation}
	where $\hookrightarrow$ stands for a continuous embedding.
	
	Given a r.i. Banach function space $X(E)$ and $0\leq s < \mathcal{L}^n(E)$ one may define the \textit{fundamental function of $X(E)$} by
	\begin{equation}\label{Jihad}
	\varphi_{X(E)}(s):= \| \chi_G\|_{X(E)}
	\end{equation}
	where $G\subset E$ is an arbitrary subset of $E$ of measure $s$.
	
	The properties of r.i.\ norms guarantee that the fundamental function is well defined. For every r.i. space $X(E)$, its fundamental function $\varphi_{X}$ is non-decreasing, $\varphi_{X(E)}(0)=0$  and $\varphi_{X}(t)/t$ is non-increasing.
	
	Given a Banach function space $X(E)$ define an associated Banach function space $X'(E)$ as a subspace of measurable functions endowed by associated norm given by
	$$
	\|f\|_{X'(E)}:=\sup\limits_{\|g\|_X\leq 1}\int_{E}fg \dx x.
	$$
	Note, that the associated space of a Banach function space is also a Banach function space and the following H\" older inequality holds
	\begin{equation}\label{TexasHoldEm}
	\int_E fg \dx x\leq \|f\|_{X(E)}\|g\|_{X'(E)}.
	\end{equation}
	Let us remind the reader that for arbitrary r.i. Banach function space we have that
	\begin{equation}\label{HarLiPo}
	g^{**}(t)\leq f^{**}(t) \quad \forall  t\in(0,\mathcal{L}^n(E))\textup{ then } \|g\|_{X(E)}\leq \|f\|_{X(E)}.
	\end{equation}
	The proof of this classical result called Hardy-Littlewood-Polya inequality may be found for instance in \cite[Theorem~4.6, Chapter~2, pg.~61]{BS}.

	Given a Banach function norm $\|\cdot\|_{X(E)}$ and a normed linear space $Y$ we shall define 
	$$
	X(E,Y):=\{f\colon E\to Y: \left(x\mapsto\|f(x)\|_Y\right)\in X(E)\}.
	$$
	Let $f\colon \er^n\to \er^k$ be a locally integrable function; we define the maximal operator of such a function by
	$$
	Mf(x):=\sup_{Q}\frac{1}{\mathcal{L}^n(Q)}\int_{Q}|f(y)|\dx y,
	$$
	where the supremum on the right-hand side is taken over all cubes $Q$ containing $x$.
	Let us also recall the \textit{Riesz-Herz equivalence} that
	\begin{equation}\label{Riesz-Herz}
	(Mf)^*(t)\approx f^{**}(t) \quad \forall t\in(0,\infty),
	\end{equation}
	with constants independent of $t$ and $f$. For proof of this result see {\cite[Theorem 3.8, Chapter~3, pg.~122]{BS}. Using together Riesz-Herz equivalence \eqref{Riesz-Herz} and Hardy-Littlewood-Polya inequality \eqref{HarLiPo} yields the following \emph{norm comparison}
\begin{equation}\label{norm-comparison}
	Mf\le Mg\qquad \Rightarrow\qquad \|f\|_{X(E)}\lesssim \|g\|_{X(E)}
\end{equation}\par
	Let $X(E)$ be a Banach function space. We say that $X(E)$ has \textit{locally absolutely continuous norm} if for any finite measure set $M\subset E$ and any function $f\in X(E)$ one has 
	$$
	M\supset M_n\rightarrow \emptyset\quad\textup{implies}\quad \|f\chi_{M_n}\|_X\rightarrow 0. 
	$$
	
	Note that if $\mathcal{L}^n(E)<\infty$ this property implies the $\varepsilon$-$\delta$-continuity of the norm. This means that for any $f\in X$ and $\varepsilon>0$ one can find $\delta>0$ such that
	\begin{equation}\label{EpsDel}
	\mathcal{L}^n(M)<\delta \textup{ implies } \|f\chi_{M}\|_X<\varepsilon.
	\end{equation}

	Let $E\subset\mathbb{R}^n$ be an open set and let $X(E)$ be a Banach function space. We define the Sobolev space over $X(E)$ by
	$$
	W^1 X(E):=\{f\in W^{1,1}_{\loc}(E,\mathbb{R}^m): \|f\|_{W^1X(E)}=\|Df\|_{X(E)}+\|f\|_{X(E)}\}
	$$
	where we use the standard operator norm to determine the size of $|Df|$.	
	
	Now follows a preparatory lemma.
	
	\prt{Lemma}
	\begin{proclaim}\label{Rozumny}
		Let $G\subset\mathbb{R}^n$ be a set of finite measure and $X(G)$ be a r.i. BFS space satisfying
		\begin{equation}\label{lim=0}
			\displaystyle{\lim_{t\rightarrow 0+}}\varphi_X(t)=0.
		\end{equation}
		Then for every $M >0$ and $\tilde{\epsilon} > 0$ there exists a $\tilde{\delta}>0$ such that for all $u\in X(G)$ with $\|u\|_{L^{\infty}(G)} \leq M$ and $\|u\|_{L^{1}(G)}<\tilde{\delta}  \mathcal{L}^n(G)$ one has $\|u\|_{X(G)}< \tilde{\epsilon}  \mathcal{L}^n(G)$.
	\end{proclaim}
	\begin{proof}
		For arbitrary $D>0$ we have
		$$
		\begin{aligned}
		\|u\|_{X(G)}&\leq \|u\chi_{\{|u|>D\}}\|_{X(G)}+\|u\chi_{\{|u|\leq D\}}\|_{X(G)}\\
		&\leq M\varphi_X\big(\mathcal{L}^n(\{|u|>D\})\big)+DC_{L^{\infty}\rightarrow X}\\
		&\leq M\varphi_{X}\left(\frac{\tilde{\delta} \mathcal{L}^n(G)}{D}\right)+DC_{L^{\infty}\rightarrow X}\\
		&=(1)+(2)
		\end{aligned}
		$$
		where $C_{L^{\infty}\rightarrow X}$ stands for the optimal constant of imbedding of $L^{\infty}(G)$ into $X(G)$. The last inequality follows from
		$$
		D\mathcal{L}^n(\{|u|>D\})\leq \|u\|_{L^1(G)}\leq \tilde{\delta}  \mathcal{L}^n(G)
		$$
		hence
		$$
		\mathcal{L}^n\big(\{|u|>D\}\big)\leq \frac{\tilde{\delta} \mathcal{L}^n(G)}{D}.
		$$
		First choose $D$ such that $(2)<\tilde{\epsilon}  \tfrac{\mathcal{L}^n(G)}{2}$. Then choose $\tilde{\delta} $ such that $(1)<\tilde{\epsilon}  \tfrac{\mathcal{L}^n(G)}{2}$.
	\end{proof}

	\begin{remark}
		If $X$ has the Lebesgue point property, we have by \cite[ Proposition 3.1]{CCPS}, that the norm is locally absolutely continuous. Therefore one has \eqref{lim=0} and Lemma~\ref{Rozumny} can be applied for such a space.
	\end{remark}
	\subsection{The reformulation of known extension results\label{EasyJet}}

	In this section our aim is to prove the following extension theorem, which will allow us to construct homeomorphisms from boundary values and gives us a useful control on their Lipschitz constant.
	
	\prt{Theorem}
	\begin{proclaim}\label{extension}
		There exists a $C>0$ such that for any $r > 0$ and any finitely piece-wise linear and one-to-one function $\varphi:\partial Q(0,r) \to\er^2$ we can find a finitely piece-wise affine homeomorphism $h:Q(0,r)\to \er^2$ such that
		\begin{equation}\label{hope}
		\|Dh\|_{L^{\infty}(Q(0,r))} \leq C_{\eqref{hope}} \|D_{\tau}\varphi\|_{L^{\infty}(\partial Q(0,r))},
		\end{equation}
		and
		$$
		h|_{\partial Q(0,r)}=\varphi.
		$$
		In \eqref{hope} the $L^{\infty}$ space on the left is with respect to the two dimensional Lebesgue measure $\mathcal{L}^2$ and the $L^{\infty}$ space on the right is with respect to the one dimensional Hausdorff measure $\mathcal{H}^1$.
	\end{proclaim}

	\begin{proof}
		The construction is precisely that of Hencl and Pratelli from \cite[Theorem 2.1]{HP} later expanded upon in \cite{ER} and \cite{C}. The construction starts by making a mapping that is not injective (only monotone) and then in the end making a small adjustment to make the mapping injective.
		
		\step{1}{A square with `good' corners}{1step1}
		
		We construct our map on a rotation of the square by $45$ degrees, we call this rotated square $Q_r$. For each $t \in (-\sqrt{2}r,\sqrt{2}r)$ denote the corresponding pair of horizontally opposite points on the rotated square $O = O_t = (|t|,t), P = P_t = (-|t|,t)$. For $t\neq 0$ we call $\widehat{OP}$ the path $[O, (0,\sgn(t)\sqrt{2}r)] \cup [ (0,\sgn(t)\sqrt{2}r), P]$ the shorter of the 2 paths from $O$ to $P$ along the boundary of $Q_r$. It was shown in \cite[Theorem 2.1, Step 1]{HP} that (up to a bi-Lipschitz transformation) one may assume that there exists a $C$ such that for any $t\in (-\sqrt{2}r,\sqrt{2}r)$ it holds that 
		$$
		\oint_{\widehat{OP}} |D_{\tau}\varphi| \dx \mathcal{H}^1\leq C \oint_{\partial Q_r}|D_{\tau}\varphi|\dx \mathcal{H}^1.
		$$
		For $t=0$ the claim holds trivially for $C=2$, for whichever path in $\partial Q_r$ we choose. The details are to be found in \cite[Theorem 2.1, Step 1]{HP}. The argument is that given $\mathcal{A}$ the set of points $P \in \partial Q_r$ such that one can find an `arc' $L$ in $\partial Q_r$ of length $2d$ symmetrically around $P$ such that
		$$
		\oint_L |D_\tau \varphi| \dx \mathcal{H}^1> 6 \oint_{\partial Q_r}|D_{\tau}\varphi|\dx \mathcal{H}^1
		$$
		is not very big. In fact using a Vitali covering of $\mathcal{A}$ one can estimate that
		$$
		\mathcal{H}^1(\mathcal{A}) \leq \sum_i 6 r_i \leq \sum_i \frac{\int_{L_i} |D_{\tau}\varphi|}{\int_{\partial Q_r}|D_{\tau}\varphi|} < 1 < \tfrac{1}{2}\mathcal{H}^1(Q_r)
		$$
		and since $\mathcal{A}$ covers less than half $\partial Q_r$ one can find a pair of opposing points in $\partial Q_r\setminus \mathcal {A}$. Any such pair can be mapped onto the north and south poles of $\partial Q_r$ by a finitely piece-wise affine map, which preserves arc length on $\partial Q$ and the bi-Lipschitz constant is independent of the choice of the pair.
		
		\step{2}{Definition of $h$ and estimate of $|Dh|$}{1step2}
		
		We define $h = \varphi$ on $\partial Q_r$. By `$h$-vertex' of $\partial Q_r$ we refer to a point $P\in \partial Q_r$ such that $D_{\tau} \varphi$ does not exist. For every $P$ a $h$-vertex of $\partial Q_r$ we define $h$ on the segment $[PO]$ (where we denote $O = (-P_1, P_2)]$) in such a way that the image of $[PO]$ in $h$ is the geodesic from $\varphi(P)$ to $\varphi(O)$ inside the closure of the bounded component of $\er^2 \setminus \varphi(\partial Q_r)$ parametrized at constant speed. The geodesic is a polyline and this way $h$ is piecewise linear on the union of $\partial Q_r$ and the added horizontal lines. Using the fact that the geodesic has length bounded by $\mathcal{H}^1\big(\varphi(\widehat{PO})\big)$, that $h$ has constant speed on the horizontal segment and that $\mathcal{H}^1\big( [P O] \big) =\tfrac{1}{\sqrt{2}}\mathcal{H}^1\big( \widehat{P O} \big) $ we get that
		\begin{equation}\label{FooFoo}
		|D_1 h|  = \frac{\mathcal{H}^1\big( h([P O]) \big)}{\mathcal{H}^1\big( [P O] \big)} \leq \sqrt{2}\oint_{\widehat{P O}} |D_{\tau}\varphi| \leq C\oint_{\partial Q_r} |D_{\tau}\varphi|
		\end{equation}
		on $[PO]$, where the last estimate is from step~\ref{1step1}.

		Let $P_1$ and $P_2$ be a pair of adjacent $h$-vertices on $\partial Q_r$ and let $O_1, O_2$ be the corresponding pair of horizontally opposing points in $\partial Q_r$. Call the strip $S$ the set between a pair of neighbouring horizontal lines, i.e. $S = \operatorname{co} \{P_1,P_2, O_1, O_2\}$, where $\operatorname{co}$ denotes the convex hull. We have defined $h$ on $\partial S$. By $F_S$ we denote the union of $h(\partial S)$ with the set disconnected from infinity by $h(\partial S)$.
		
		We simply separate $S$ into triangles $A_1A_2A_3$ with $[A_1A_2]$ lying on one of the horizontal segments and $[A_2A_3]$ is vertical and all of $A_1,A_2,A_3$ lie on one of the horizontal segments $[P_1O_1]$ or $[P_2O_2]$. Since the images of the corners $A_1A_2A_3$ in $h$ have been defined there is exactly one affine map sending $A_i$ onto $h(A_i)$ and we define $h$ as this affine map on each triangle $\co\{A_1,A_2,A_3\}$. Then immediately from \eqref{FooFoo} we get
		\begin{equation}\label{GooGoo}
		|D_1h| \leq C\oint_{\partial Q_r} |D_{\tau}\varphi| \leq C \|D_{\tau}\varphi\|_{L^{\infty}(\partial Q_r)}.
		\end{equation}
		A careful analysis of the geometry of the geodesics, fully exposed in \cite[Theorem 2.1, Step 4,5,6]{HP} proves that these affine images of $A_1A_2A_3$ lie in $F_S$ and provides an estimate on $|h(A_2) - h(A_3)|$. We summarize the steps from \cite[Theorem 2.1, Step 4,5,6]{HP} in the following lemma. The key length estimate \eqref{BSS} is \cite[estimate (2.4)]{HP}.

		\begin{lemma}\label{verticals}
			Let $P_1$ and $P_2$ be a pair of adjacent $h$-vertices on $\partial Q_r$ and let $O_1, O_2$ be the neighbouring pair of horizontally opposing points in $\partial Q_r$. Call $S = \operatorname{co} \{P_1,P_2, O_1, O_2\}$ and $F_S$ is the union of $h(\partial S)$ with the part(s) of the plane it disconnects from infinity. Then $h([A_2A_3]) \subset F_S$ and
			\begin{equation}\label{BSS}
			\H^1(h([A_2A_3])) \leq \max\{\H^1\big(\varphi([P_1P_2])\big), \H^1\big(\varphi([O_1O_2])\big)\}.
			\end{equation}
		\end{lemma}
		
		But then from \eqref{BSS}, on the triangle $\co\{A_1, A_2, A_3\}$ we have
		\begin{equation}\label{MooMoo}
		\begin{aligned}
		|D_2h|  &= \frac{\mathcal{H}^1(h([A_2A_3]))}{\mathcal{H}^1([A_2A_3])}\\
		& \leq \sqrt{2}\frac{\max\{\H^1\big(\varphi([P_1P_2])\big), \H^1\big(\varphi([O_1O_2])\big)\}}{|P_1 - P_2| }\\
		& \leq \sqrt{2}\|D_{\tau}\varphi\|_{L^{\infty}(\partial Q_r)}
		\end{aligned}
		\end{equation}
		This holds on all triangles $\co\{A_1, A_2, A_3\}$ in all strips $S$ between neighbouring horizontal lines. The triangle at the north and south pole of $Q_r$ is estimated similarly.
		
		\step{3}{Injectification of $h$}{1step3}
		
		The final step is to replace the image of the lines $[PO]$ with something very close to the geodesics but so that the images never meet. It suffices to shift the image of each $[PO]$ in $h$ slightly inside $\partial Q_r$ at leach corner of $\varphi(\partial Q_r)$. As long as this change is very small then all of the estimates remain intact (see also \cite[Theorem 2.1, Step 10]{HP}).
		
		The claim \eqref{hope} follows from \eqref{GooGoo} and \eqref{MooMoo}.
	\end{proof}
	
	The following Corollary is an immediate result of Theorem~\ref{extension}. It is the fact that we get an $L^{\infty}$ bound from the $L^{1}$ space that our approach works in the relatively general setting.
	
	\prt{Corollary}
	\begin{proclaim}\label{RePara}
		There exists a constant $C>0$ such that for every $r > 0$ and $\varphi:\partial Q(0,r)\to\er^2$ finitely piece-wise linear and one-to-one function with $|D_{\tau}\varphi|$ constant on each side of $Q_r = Q(0,r) \subset \er^2$, there exists a piece-wise affine homeomorphism $h:Q_r\to \er^2$ such that
		\begin{equation}\label{SuperMan}
			\|Dh\|_{L^{\infty}(Q_r)} \leq \frac{4C_{\eqref{hope}}}{r} \int_{\partial Q_r}|D_{\tau}\varphi|\, d\H^1\,.
		\end{equation}
		and 
		$$
		h|_{\partial Q_r}=\varphi.
		$$
	\end{proclaim}
	\begin{proof}
		If $|D_{\tau}\varphi|$ is constant on sides, then $\|D_{\tau}\varphi\|_{L^{\infty}(\partial Q_r)}\leq \tfrac{4}{r} \int_{\partial Q_{r}}|D_{\tau}\varphi|\, d\H^1$.
	\end{proof}

	Further we will reformulate \cite[Theorem 3.7]{C} to fit in with our current setting better. The question is how to approximate a map which is close to a degenerate linear map $\Phi$ (up to a rotation in the pre-image we may assume that $\Phi = \left(\begin{matrix}
	d,0\\
	0,0\\
	\end{matrix}\right)$).

	\prt{Theorem}
	\begin{proclaim}\label{extension2}
		Let $d> \delta >0$, let $r_0 \in(0,1)$ and let $Q$ be a convex set and the image of $[0,r_0]^2$ in a $2$-bi-Lipschitz mapping which is equal to an affine mapping on $\co\{(0,0), (0,r_0),(r_0,0)\}$ and $\co\{(r_0,r_0), (0,r_0),(r_0,0)\}$. Then for every $\varphi:\partial Q\to\er^2$ finitely piece-wise linear and one-to-one mapping with
		\begin{equation}\label{bound1}
		\int_{\partial Q} \Big|D_{\tau} \varphi(t)-\left(\begin{matrix}
		d,0\\
		0,0\\
		\end{matrix}\right)\tau\Big| \; d\H^1(t)< \delta r_0 ,
		\end{equation}
		and $\|D_{\tau}\varphi\|_{L^{\infty}(\partial Q)}\leq d+2\delta$, where $\tau$ is the unit tangential vector to $\partial Q$,
		there exists a finitely piece-wise affine homeomorphism $g:Q\to \er^2$ and a set $W\subset Q$ such that $\mathcal{L}^2(Q\setminus W)<C\delta r_0^2$, $g=\varphi$ on $\partial Q$,
		\begin{equation}\label{CoolAssEquation}
		\begin{aligned}
		\Big\| Dg(x)-\left(\begin{matrix}
		d,0\\
		0,0\\
		\end{matrix}\right)\Big\|_{L^{\infty}(Q)} &< C(d+1),\\
		\Big\| Dg(x)-\left(\begin{matrix}
		d,0\\
		0,0\\
		\end{matrix}\right)\Big\|_{L^1(W)} &< C\delta r_0^2  \text{ and } \\
		\Big\| Dg(x)-\left(\begin{matrix}
		d,0\\
		0,0\\
		\end{matrix}\right)\Big\|_{L^{\infty}(W)} &< 3d.
		\end{aligned}
		\end{equation}
		
	\end{proclaim}
	
	\begin{proof}
		\step{1}{Definition of $W$ such that $\mathcal{L}^2(Q\setminus W)<C\delta r_0^2$}{SDefW}
		
		Since $Q$ is the convex 2-affine image of a square, it makes sense to talk about its vertices. Calling $A_i$ the vertices of $Q$ and $A = \tfrac{1}{4}\sum_{i=1}^4 A_i$ the centre of $Q$, we define
		$$
		B_i = B_i(\delta) = A_i+ 10 \delta r_0 (A - A_i).
		$$
		We call $W = W(\delta) = \co\{B_1, B_2, B_3, B_4\}$. Then it follows that $|Q\setminus W(\delta)| \leq C\delta r^2_0$. 
		
		\step{2}{Definition of a curve $\gamma$ that goes from close to one end of $\varphi(\partial Q)$ to the other}{SDefGamma}

		We find the point $P \in \partial Q$ such that $P$ is a point where $x^1$ achieves its minimum for $x\in \partial Q$ and $O \in \partial Q$ is a point where $x^1$ achieves its maximum for $x\in \partial Q$. We call $\widehat{PO}$ the shortest path in $\partial Q$ connecting $P$ and $O$. If there are two paths then choose either of them. From \eqref{bound1} we have that
		\begin{equation}\label{BooBoo}
		\int_{\widehat{PO}} \Big|D_{\tau} \varphi(t)-\left(\begin{matrix}
		d,0\\
		0,0\\
		\end{matrix}\right)\tau\Big| \dx \H^1(t)< \delta r_0.
		\end{equation}
		We define a piecewise linear path in the interior of $\varphi(\partial Q)$ following $\varphi(\partial Q)$ from its start very close to $\tP = \varphi(P)$ to its end very close to $\tO = \varphi(O)$. Call this path $\gamma$. The idea how to do this is simple. At every vertex of $\varphi(\partial Q)$ we bisect the interior angle (the bisector goes inside the bounded component of $\er^2 \setminus \varphi(\partial Q)$) and place a point on the bisector with distance to the vertex as small as required. Then we form a polyline by connecting the points we constructed in the same order as the vertices (see \cite[Lemma 3.1]{C} for a precise construction). By placing the points we constructed on the bisectors very close to their respective vertices we ensure that the length of the curve differs from the length along the boundary by a number as small as required.
		
		\step{3}{Definition of $g$ on $W$ and estimates}{SDefGW}
		
		We define the function $l : W \to \er ^2$ as follows. Let $\gamma$ denote the constant speed parametrization of the curve $\gamma$ from $[0,1]$. Then define
		$$
		l(x,y) = \gamma(\frac{x-P_1}{O_1 - P_1}).
		$$
		Now it is not hard to check using \eqref{BooBoo} that
		\begin{equation}\label{PooPoo}
		\int_{W} \Big|D l-\left(\begin{matrix}
		d,0\\
		0,0\\
		\end{matrix}\right)\Big| \dx \mathcal{L}^2< C \delta r^2_0.
		\end{equation}
		The construction of $g$ on $W$ is a slight modification of $l$ so that the map is injective. This is just a question of decomposing $W$ into quadrilaterals of type $W\cap \{x; a<x<b\}$ and instead of mapping them onto the curve $\gamma$ we map them onto a tiny tubular neighbourhood of $\gamma$. On each of these quadrilaterals we use a bi-affine map. The details are in \cite[Theorem 3.7, step 3]{C}. Most importantly we get the estimate on the derivative
		\begin{equation}\label{ZooZoo}
		\int_{W} \Big|D g-\left(\begin{matrix}
		d,0\\
		0,0\\
		\end{matrix}\right)\Big| \dx \mathcal{L}^2< C \delta r^2_0
		\text{ and } \Big|D g-\left(\begin{matrix}
		d,0\\
		0,0\\
		\end{matrix}\right)\tau\Big| \leq 3d
		\end{equation}
		immediately from \eqref{PooPoo}, the fact that $|D_1l|$ is constant and $g$ is as close to $l$ in $W^{1,\infty}$ as we like.

		\step{4}{Extension of $g$ onto $Q\setminus W$ and estimates}{SDefGQ}
		
		The remaining set, $Q\setminus W$, is a tubular neighbourhood of $\partial Q$ in the direction inside of width approximately $\delta r_0$. Therefore it can be divided into small quadrilaterals all 2-biLipschitz equivalent with a square, which we denote by $U_i$.
		
		On each $\partial U_i$ we define a piecewise linear mapping $\varphi_i$ as follows. On $\partial U_i\cap \partial W$ put $\varphi_i = g$ and on $\partial U_i\cap \partial Q$ put $\varphi_i = \varphi$. Then we need to define $\varphi$ on the `radial' segments. This equates to the following task: define mutually non-intersecting polylines from points on $g(\partial W)$ to corresponding points on $\varphi(\partial Q)$ lying inside the bounded component of $\er^2 \setminus (\varphi(\partial Q) \cup g(W))$. This is obviously a task that has a solution. What is more important is that we have a bound on the length of each polyline, which was proved in \cite[Theorem 3.7, step 5, 6]{C} and which we summarize in the following lemma. The key length estimate \eqref{BSSS} is (3.27) of \cite{C}.
		
		\begin{lemma}\label{PolylinesLemma}
			Let $0<\delta < \tfrac{1}{16}$, $w\in \partial W(\delta)$ and $x \in \partial Q$ with $|w - x|< 4\delta r_0$ then there exists a polyline $p_{x,w}$ connecting $\u{x} = g(x) = \varphi(x)$ with $\u{w} = g (w)$ lying inside the bounded component of $\er^2 \setminus (\varphi(\partial Q) \cup g(W))$ (except for the endpoints) and it holds that
			\begin{equation}\label{BSSS}
			\mathcal{H}^1(p_{x,w}) \leq (d+1) |w-x| + C \delta r_0
			\end{equation}
			with $C$ independent on the choice of $x$ and $w$.
		\end{lemma}
		\begin{proof}[Proof of Lemma~\ref{PolylinesLemma}]
			For simplicity of the argument let us assume that $x$ and $w$ are points on the top of $W$ and $Q$ respectively. In the case that the points lie on the bottom or on a side that is very close to being vertical then the argument is the basically the same and we emphasize the difference when it is relevant.
			
			We start by defining a tentative polyline that we later alter to ensure injectivity. By \cite[Lemma 3.5]{C} it suffices to follow $g(\partial W)$ for a distance of at most $C\delta r_0$ till we find a point $\u{w}' \in g(W)$ such that the that the vertical segment going upwards (downwards for points on the bottom of $W$ resp. $Q$ and either case holds for points on sides very close to vertical) from $\u{w}'$ does not intersect $g(W)$ but only $g(\partial Q)$ at a point, which we call $\u{x}'$. The distance from $\u{x}'$ to $\u{x}$ along $g(\partial Q)$ is estimated from above by $|\u{x} - \u{x}'| + \delta r_0$ using \eqref{bound1}. But $|\u{x} - \u{x}'| \leq d|w-x| + C\delta r_0$. Therefore there is a polyline from $\u{x}$ to $\u{w}$ in the closure of the interior of $g(\partial Q)$ not intersecting $g(W^{\circ})$ with length $Cd\delta r_0$ with $C$ independent on the choice of $\u{x}$ and $\u{w}$. The polyline can be moved inside the bounded component of $\er^2 \setminus (\varphi(\partial Q) \cup g(W))$ except for the endpoints making it only $\epsilon$ longer with $\epsilon$ as small as we like.
		\end{proof}
		
		It is a standard technique (used to show the uniqueness of shortest paths in $\er^2$) that allows us to show that if two polylines from Lemma~\ref{PolylinesLemma} intersect then one can redefine them so that they do not touch and the length estimate still holds.
		
		Now we define $\varphi_i$ as the constant speed parametrization of $p_{x,w}$ on each `radial' segment of a `square' $U_i$. Then we calculate that
		$$
		D_{\tau} \varphi_i(t) \leq \begin{cases}
		2d \quad &t \in \partial Q\\
		2d \quad &t \in \partial W\\
		C(d+1) \quad &t \text{ on radial segments of } \partial U_i
		\end{cases}
		$$
		because $\mathcal{H}^1(p_{x,w}) \leq  C (d+1)\delta r_0$ and the length of the radial segments are approximately $C\delta r_0$. Then applying Theorem~\ref{extension} on each $U_i$ we define $g$ on $Q\setminus W$ such that $\|Dg\|_{L^{\infty}(Q)} \leq C(d+1)$. The fact that $g = \varphi$ on $\partial Q$ is immediate from our definition. 
	\end{proof}
	
	\begin{lemma}\label{Approx}
		Let $X(\Omega)$ be a r.i. BFS space  such that $$\lim_{t\rightarrow 0} \varphi_{X(\Omega)}(t)=0.$$ Let $\hat{f}\in W^1X(\Omega,\mathbb{R}^2)$ be a locally finite piece-wise affine homeomorphism. Then for every $\varepsilon>0$ there exists a diffeomorphism $\tilde{f}$ such that 
		$$
		\|D\tilde{f}-D\hat{f}\|_{X(\Omega)}<\varepsilon
		$$
		and
		$$
		\|\tilde{f}-\hat{f}\|_{L^{\infty}(\Omega)}<\varepsilon.
		$$
		Moreover, if $\hat{f}$ is continuous up to the boundary of $\Omega$, then $\tilde{f}$ can be chosen to be continuous up to the boundary of $\Omega$ and $\tilde{f}= \hat{f}$ on $\partial \Omega$.
		
	\end{lemma}
	\begin{proof}
		We denote by $\hat{f}$ the countably and locally-finite piece-wise homeomorphism and by $\tilde{f}$ a diffeomorphism constructed as in \cite{MP}. The claim that
		$$
		\|\tilde{f}-\hat{f}\|_{L^{\infty}(\Omega)}<\varepsilon
		$$
		is part of the claim of Theorem~A in \cite{MP}. We only need to show that the arguments in \cite{MP} extend to handling the norm $\| \cdot\|_{X(\Omega)}$ instead of $\| \cdot\|_{L^p(\Omega)}$ for the derivatives. We do not focus on results for the inverse or the determinant.
		
		The construction from \cite{MP} works on a small neighbourhood of the the edges and vertices of each triangle. Around edges of triangles (in \cite{MP} referred to as $Z_3$), but distant from the vertices it suffices to take an appropriate, smooth convex combination of the two affine maps, which meet at the edge. This gives a diffeomorphism on a (small tubular) neighbourhood of the edge, which coincides with $\hat{f}$ away from the edge.
		
		On small disks close to vertices one does two separate steps. One works on an outer annulus (in \cite{MP} referred to as $Z_2$) and a disk inside (in \cite{MP} referred to as $Z_1$). Assume that we have already smoothed close to the edges ending at the vertex. On the outer annulus one does a convex combination (with respect to $r$) with a smooth map, which expressed in polar coordinates $(r,\theta)$ has the same `argument' $\theta$ as the piecewise affine map (after smoothing near the edges) but brings all the points closer to the vertex in such a way that it sends circles onto circles. On the disk inside the outer annulus we smoothly transition to a small multiple of identity. This is possible since circles are sent to circles and it suffices to appropriately rotate the circles and smoothly take the angular speed (of the map expressed in polar coordinates in the preimage and image) to 1. In this case we are equal to a translation plus a small multiple of the identity near each vertex.
		
		In each case the neighbourhoods of the edges ($Z_3$) and vertices ($Z_1, Z_2$) can be made as small as we like. On the rest of the space (in \cite{MP} referred to as $Z_4$) we have $\tilde{f}=\hat{f}$. To reiterate \cite{MP} we quote the final paragraph of their Proof of Theorem~A they say ``Finally, the set [$\{y : \tilde{f}(y) \neq \hat{f}(y)\}$] can be made as small as we wish, by decreasing the constants $\delta_a$, $\delta_b$ and $\delta_c$ as needed''.
		
		The mapping $\hat{f}$ is piecewise affine on $\Omega$, which means there is a set of triangles $\left\{T_i\right\}_{i\in I}$ covering $\Omega$ and $\hat{f}$ is affine on each $T_i$. For each $T_i$, by construction we have that there exists a constant $M_i = 13 \max_{t\in J(x)} |D\hat{f}(t)|$, where $J(x)$ is the union of $T_i$ containing $x$ with all triangles $T_j$ which intersect $T_i$ (take $q =1$ in \cite[Theorem A]{MP}) such that
		$$
		\|D \hat{f} \|_{L^\infty(T_i)}\leq M_i, \quad \|D\tilde{f}\|_{L^\infty(T_i)}\leq M_i.
		$$
		Obviously,
		$$
		\| D \hat{f}- D\tilde{f}\|_{X(\Omega)}\leq \sum_{i\in I}\big \|\left(  D\hat{f}- D\tilde{f}\right)\chi_{T_i}\big\|_{X(\Omega)}
		$$	
		Moreover, for each $i\in I$
		$$
		|D\hat{f}- D\tilde{f}|\chi_{T_i} \leq \left( |D\hat{f}|+ |D\tilde{f}| \right)\chi_{T_i\setminus Z^i_4}\leq 2 M_i \chi_{T_i\setminus Z^i_4}
		$$
		where we denote by $Z^i_4$ the intersection of the set where $\hat{f} = \tilde{f}$ (in \cite{MP} referred to as $Z_4$) with $T_i$.
		
		Therefore
		$$
		\big\| \left(  D\hat{f}- D\tilde{f}\right)\chi_{T_i}\big\|_{X(\Omega)} \leq 2M_i\,\varphi_{X(\Omega)} \left(\mathcal{L}^2(T_i \setminus Z^i_4)\right).
		$$
		We can assume that $\mathcal{L}^2(T_i \setminus Z^i_4)$ is so small that
		$$
		\varphi_{X(\Omega)} \left(\mathcal{L}^2(T_i \setminus Z^i_4)\right) < \frac{\varepsilon}{2M_i} 2^{-i}
		$$
		and thus we conclude
		$$
		\| D\hat{f}- D\tilde{f}\|_{X(\Omega)}\leq \varepsilon.
		$$
	\end{proof}

	\subsection{Preliminary results on grids and approximations on grids\label{SuperBoring}}
	
	Let us remind the reader that by \textit{dyadic squares} we mean the family $\mathcal{D}=\bigcup_{k\in\mathbb{Z}}\mathcal{D}_k$, where
	$$
	\mathcal{D}_k =\{Q((2^{k}, 2^k) + z 2^{k+1}, 2^{k}); z\in \mathbb{Z}^2\},
	$$
	where the square $Q((a,b),r) = [a-r,a+r]\times[b-r,b+r]$ is closed. Given $v\in \er^2$, by  $\mathcal{D}_k^v$ we denote
	$$
	\mathcal{D}_k^v=\{v+Q; Q\in\mathcal{D}_k\}
	$$
	
	For a domain $\Omega$ we say that an open set $G$ \textit{separates components of $\partial\Omega$} if any continuous curve connecting different components of $\partial\Omega$ has to intersect $G$. We say that the domain $\Omega\subset\mathbb{R}^2$ is \textit{finitely connected} if $\er^2\setminus\Omega$ has finite number of components.
	
	\begin{lemma}\label{Segregation}
		Let $\Omega \subset \er^2$ be a finitely connected bounded domain and let $f \in W^{1,1}_{\text{loc}}(\Omega, \er^{2})$ be a homeomorphism. Then there exists a strictly increasing sequence of sets $\Omega_k\Subset\Omega_{k+1}\Subset\Omega$ such that $\Omega = \bigcup_k\Omega_k$, also $\Omega_k$ separates components of $\partial \Omega$ and $\partial\Omega_k$ is piece-wise linear and parallel to coordinate axes and $\partial \Omega_{k}$ has the same number of components as $\partial\Omega$.
		
		Further, for every $k\in\mathbb{N}$ and $\epsilon_k, \delta_k > 0$ and any set $A_k$ with $\mathcal{L}^2(A_k)<\tfrac{\delta_k}{32}$ there exists an $m_k \in \en$, and $v_k \in Q(0,2^{-m_k-1})$ and a collection of $K_k$ shifted dyadic squares $\{{Q}_i^k= Q(c_i+v_k, 2^{-m_k}) \in \mathcal{D}_{-m_k}^{v_k}\}_{i=1}^{K_k}$ such that
		\begin{enumerate}
			\item[$i)$] $Q_i^k\Subset \Omega_k \setminus \Omega_{k-1}$, and $\bigcup_{i=1}^{K_k} {Q}_i^k$ has no holes (i.e. in each component of $\Omega_k \setminus \Omega_{k-1}$  there are precisely 2 components of $ \Omega_k \setminus \bigcup_{i=1}^{K_k} {Q}_i^k$),
			\item[$ii)$] it holds that
			$$
			\mathcal{L}^2\Big(\Omega_k\setminus \Big[\Omega_{k-1}\cup \bigcup_{i=1}^{K_k} {Q}_i^k\Big]\Big)  <\delta_k,
			$$
			\item[$iii)$] for each $1\leq j \leq K_k$ the square ${Q}_j^k$ shares at least two of its sides with other squares of $\{{Q}_i^k\}_{i=1}^{K_k}$,
			\item[$iv)$] the set $\partial\bigcup_{i=1}^{K_k}{Q}_i^k$ consists of segments and there is a one-to-one correspondence between the endpoints of these segments and vertices of $\partial(\Omega_k\setminus \Omega_{k-1})$ i.e. for every $X$ vertex of $\partial(\Omega_k\setminus \Omega_{k-1})$ there is exactly one vertex of $\partial\bigcup_i{Q}_i^k$ in $B(X,2^{4-m_k})$ and vice versa,
			\item[$v)$] for every $x\in \partial\bigcup_{i=1}^{K_k}{Q}_i^k$ we have
			$$
				6 \cdot 2^{-m_k}\leq \dist_{\infty}(x, \partial(\Omega_k\setminus \Omega_{k-1})) \leq 8 \cdot 2^{-m_k}\,,
			$$
			\item[$vi)$] there exists a set of indexes $B_k$ such that $\mathcal{L}^2(\bigcup_{i\in B_k} {Q}^k_{i}) < \delta_k$ and moreover for every $i \notin B_k$,
			$$
			c_i+v_k \notin A_k
			$$
			and
			\begin{equation}\label{ThisOne}
				\begin{aligned}
			\|f(x) - f(c_{i}+v_k) - Df(c_{i}+v_k)(x-c_{i}-v_k) \|_{L^{\infty}(2Q_{i} ^k)} &< \epsilon_k 2^{-m_k-1},\\
			\oint_{2Q_{i}^k} |Df(x)-Df(c_{i}+v_k)| \dx\mathcal{L}^2(x) &< \tfrac{\epsilon_k}{4},\\
			\tfrac{1}{\mathcal{L}^2(Q_i^k)}\|(Df - Df(c_i  +v_k))\chi_{2Q_i^k}\|_{X(\Omega)} &< \epsilon_k
			\end{aligned}
			\end{equation}
			\item[$vii)$] for every $i=1,\dots, K_k$ we have $\diam(f(2Q_i^k))\leq\epsilon_k$.
		\end{enumerate}
	\end{lemma}

	\begin{proof}
		
		\begin{figure}
			\centering
			\includegraphics[width=10cm]{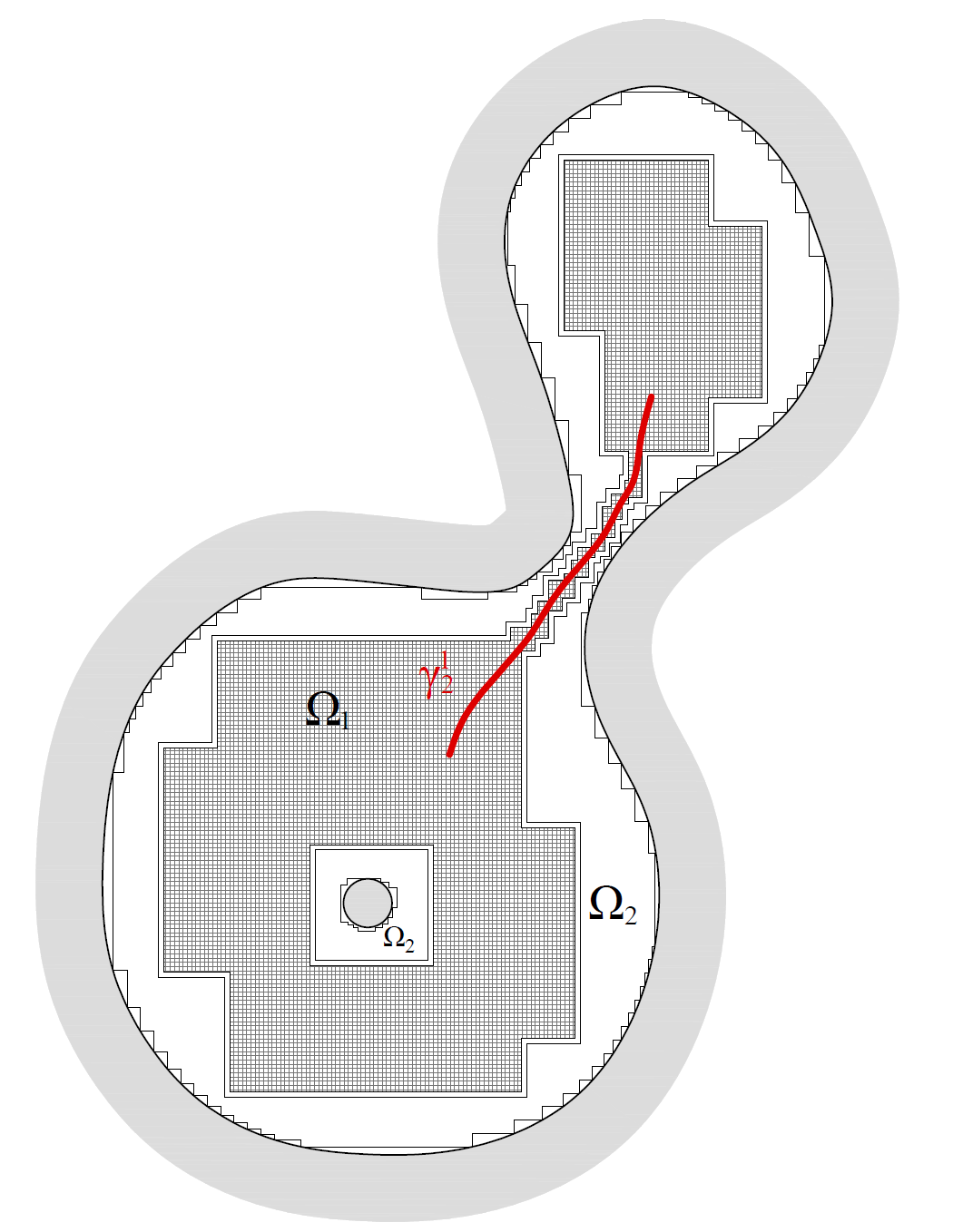}
			\caption{The sets $\Omega_1$ and $\Omega_2$ inside $\Omega$ a doubly connected domain. The squares $Q_i^1$ are also depicted and the curve $\gamma_2^1$.}\label{fig:Seg}
		\end{figure}
		Our choice of the sets $\Omega_k$ is inductive. We start by choosing $\Omega_1$. After we have done that we assume that we have an $\Omega_{k-1}$ and its corresponding parameter $l_{k-1}$ and then choose $\Omega_k$ based on the previous set.
		
		For every $l \in \en$ we define $\mathcal{W}_{-l}$ as the set of squares $Q\in \mathcal{D}_{-l}$ such that $16Q\subset \Omega$. As $l$ tends to infinity these sets fill $\Omega$ and because $\Omega$ is finitely connected there exists an $l_0$
		such that the set
		$$
		\tilde{\Omega}_1=\bigcup_{Q\in\mathcal{W}_{-l_0}}Q
		$$
		separates components of $\partial\Omega$ (i.e.\ any continuous curve connecting different components of $\partial\Omega$ intersect $\tilde{\Omega}_1$). The set $\tilde{\Omega}_1$ is not necessarily connected. If $\tilde{\Omega}_1$ is connected, we set $\Omega_1=\tilde{\Omega}_1$. In the opposite case, we argue as follows. Assume that there are $N_1$ components, we choose a point $y_n^1$ in each of the components. Each pair $y_1^1$ and $y_n^1$, $n=2,\dots, N_1$ are path-wise connected in $\Omega$. Call $\gamma_n^1$ a path connecting $y_n^1$ with $y_1^1$ inside $\Omega$. Now set
		$$
		l_1:=\min\Big\{l\in\mathbb{N}\ :\ l>l_0\text{ and } \gamma_{j}^1\subset \bigcup\limits_{\mathcal{W}_{-l}}Q\text{\quad for }2\leq j \leq {N}_1\Big\}.
		$$
		Call $\tilde{C}_1$ the union of all $Q$ such that
		$$
		Q\in \Big\{\tilde{Q}\in \mathcal{D}_{-l_1}\ :\ 4\tilde{Q}\cap \gamma^1_j \neq \emptyset\text{\quad for }2\leq j \leq {N}_1\Big\}
		$$
		then we define the set 
		$$
		C_1 = \tilde{\Omega}_1\cup \tilde{C}_1.
		$$
		Also we call $D_1$ the union of the squares $Q$ such that
		$$
		Q\in\Big\{\tilde{Q}\ :\ \tilde{Q}\in \mathcal{D}_{-l_1} \text{ and } \tilde{Q} \text{ is disconnected from } \partial \Omega \text{ by } C_1\Big\}.
		$$
		By $A^{\circ}$ we denote the topological interior of $A$. Note that the set
		$$
		\Omega_1:=(C_1\cup D_1)^{\circ}
		$$
		is a finitely connected domain and every component of $\er^2\setminus\Omega_1$ contains exactly one component of $\partial \Omega$.\\

		Continuing to the induction step, let us suppose that we have defined $l_{k-1}$ and $\Omega_{k-1}$. Set
		$$
		\tilde{\Omega}_k = \bigcup_{Q\in\mathcal{W}_{-l_{k-1} - 5}}Q.
		$$
		If $\tilde{\Omega}_k$ is not connected then continue as before. Call $N_{k}$ the number of components of $\tilde{\Omega}_k$ and let $\gamma_n^k$, $n=2,3,\dots, N_k$ be paths in $\Omega$ connecting each component of $\tilde{\Omega}_k$ to one of the given components. We find an $l_k > l_{k-1} + 5$ such that we can cover each $\gamma_n^k$ with dyadic squares from $\mathcal{W}_{-l_k}$. We call $\tilde{C}_k$ the union of squares $Q$ such that
		$$
		Q \in \Big\{\tilde{Q}\in \mathcal{D}_{-l_k}\ :\ 4\tilde{Q}\cap \gamma^k_j \neq \emptyset\text{\quad for }2\leq j \leq {N}_k\Big\}
		$$
		and use this to define $C_k$ as follows
		$$
		C_k = \tilde{\Omega}_{k}\cup \tilde{C}_k.
		$$
		Further we define $D_k$ as the union of the squares $Q$ such that
		$$
		Q\in \Big\{Q\in \mathcal{D}_{-l_k}\ :\ Q \text{ is disconnected from } \partial \Omega \text{ by } C_k\Big\}
		$$
		for squares that were surrounded by $C_k$ in order to define
		$$
		\Omega_k:=C_k\cup D_k.
		$$
		Thus we define $\Omega_k$, $k\in \en$ inductively. The set $\Omega_k$ is a union of some squares in $\mathcal{D}_{-l_k}$ with distance to the boundary of at least $2 \cdot 2^{-l_k}$. On the other hand $\tilde\Omega_{k+1}$ contains the union of all squares in $\mathcal{W}_{-l_k-5}$ and this contains all points with distance to the boundary of $2^{-l_k}$ and therefore is a strictly larger set than $\mathcal{W}_{-l_k}$. Hence it is easy to deduce that $\Omega_k\Subset\tilde\Omega_{k+1}\subset\Omega_{k+1}$ and so this decomposition satisfies our requirements for $\Omega_k$.
		
		Now let us choose $k\in \en$ and construct the collection of cubes $\{Q_i^k\}$ in the claim. Either $\Omega$ is simply connected and $\Omega_k$ has one component, or $\Omega$ is multiply connected and $\Omega_1$ has one component and for all $k\geq 2$ the number of components of $\Omega_k\setminus \Omega_{k-1}$ equals the number of components of $\partial \Omega$. In the latter case it suffices to consider each component separately and so with respect to this fact we may assume that $\Omega_k\setminus \Omega_{k-1}$ is connected.
		
		We have $k$ fixed; for each $m \in \en$ we call
		$$
		{U}_{m}^k = \{Q \in \mathcal{D}_{-m}\ :\ Q \subset \Omega_k\setminus \Omega_{k-1}\}
		$$
		where $m \geq l_k+8$. Firstly notice that the map $f$ is uniformly continuous on $\Omega_k$ and so for $\epsilon_k$ there exists an $m'$ such that $\diam f(Q)< \epsilon_k$ for all $m\geq m'$ and $Q\in {U}_{m}^k$. By this condition we get point $vii)$.
		
		Calling
		$$
		D = \mathcal{H}^1(\partial(\Omega_k\setminus\Omega_{k-1})) + \#\{\text{vertices of } \partial(\Omega_k\setminus\Omega_{k-1})\}
		$$
		we get the existence of an $m''$ such that
		\begin{equation}\label{IncyWincy}
		2^{3-m}D \leq \delta_k
		\end{equation}
		for all $m\geq m''$. This will be crucial for getting $ii)$.
		
		Our next step will be to shift the squares to guarantee that $vi)$ holds, then we exclude the squares too close to $\partial(\Omega_k \setminus \Omega_{k-1})$ to give $iv)$ and $v)$. The other properties will follow quickly.
		
		Almost every $ a\in \Omega_k\setminus \Omega_{k-1}$ is a point of differentiability of $f$ (for example see~\cite[Lemma~A.28]{HK}) and a Lebesgue point of the derivative of $f$ in $X$ (it follows from \eqref{RISANDWICH} that it is also a Lebesgue point of $Df$ in classical sense). We have
		\begin{equation}\label{LPOD}
		\begin{aligned}
		I_{m}(a):={}&2^{m}\|f(x) - f(a) - Df(a)(x-a) \|_{L^{\infty}(Q(a,2^{1-m}))}\\
		&\qquad + \oint_{Q(a,2^{1-m})} |Df(x)-Df(a)| \dx\mathcal{L}^2(x)\\ &\qquad+ \frac{1}{\mathcal{L}^2(Q(a,2^{1-m}))}\|\left(Df-Df(a)\right)\chi_{2Q(a,2^{1-m})}\|_{X(\Omega)}\longrightarrow 0\text{\quad as } m\to \infty 
		\end{aligned}
		\end{equation}
		for almost every $a\in \Omega_{k}\setminus\Omega_{k-1}$ (because $\Omega_k\Subset \Omega$ we may assume that $Q(a,2^{1-m})\subset \Omega$ for all $m$ large enough). Therefore there exist an $m'''$ and a set $E_k$ such that for all $m\geq m'''$
		$$
		\left\{I_m\geq \tfrac{\varepsilon_k}{4}\right\}\subset E_k \text{\quad  and\quad } \mathcal{L}^2(E_k)<\tfrac{\delta_k}{32}.
		$$
		Now we choose $m_k = \max\{l_k+8, m', m'', m'''\}$ and denote $U_{k} = U_{m_k}^k = \{\tilde{Q}_1^k, \tilde{Q}_2^k, \dots, \tilde{Q}_{O_k}^k\}$, where $\tilde{Q}_i^k = Q(c_i, 2^{-m_k})$. Define
		$$
		\psi(x) = \sum_{i=1}^{O_k} (\chi_{E_k}(x+c_i)+ \chi_{A_k}(x+c_i)) \text{\quad for } x\in Q(0,2^{-m_k-1})
		$$
		and then
		$$
		\begin{aligned}
		\int_{Q(0,2^{-m_k-1})} \psi \dx \mathcal{L}^2 &=  \sum_{i=1}^{O_k} \int_{Q(c_i,2^{-m_k-1})} (\chi_{E_k} + \chi_{A_k}) \dx \mathcal{L}^2\\
		& \leq   \int_{\Omega} (\chi_{E_k} + \chi_{A_k}) \dx \mathcal{L}^2\\
		&  \leq   \mathcal{L}^2(E_k) + \mathcal{L}^2(A_k) \\
		& \leq \frac{\delta_k}{16}.
		\end{aligned}
		$$
		Therefore
		\begin{equation}\label{INTPR}
		\oint_{Q(0,2^{-m_k-1})} \psi \dx \mathcal{L}^2 \leq \frac{\delta_k2^{2m_k}}{16} = \delta_k2^{2m_k-4}
		\end{equation}
		and therefore we can find a $v_k \in Q(0, 2^{-m_k-1})$ such that
		\begin{equation}\label{NotBad}
		\psi(v_k) \leq \delta_k 2^{2m_k-4}.
		\end{equation}
		
		Now, set
		$$
		a_i=c_i+v_k\quad \textup{and}\quad Q^{ k}_{ i}=Q(a_i,2^{-m_k}) =v_k+\tilde Q^k_i\,.
		$$
		Let 
		$$
		B_k:=\left\{i:I_{m_k}(a_i)\geq \frac{\varepsilon_k}{4}\right\}\cup\{i: a_i \in A_k\}
		$$
		be the set of indices of the bad squares. Now, note that since \eqref{NotBad} holds, we have $\card( B_k)\leq \delta_k 2^{2m_k-4}$ and thus
		$$
		\mathcal{L}^2\left(\bigcup\limits_{i\in B_k} Q^i_k\right)\leq \sum\limits_{i\in B_k}\mathcal{L}^2(Q^i_k)\leq 2^{2-2m_k}\delta_k 2^{2m_k-4}=\tfrac{1}{4}\delta_k.
		$$
		In summary we have the following, each square $Q(a_i, 2^{-m_k})$ either has $i\in B_k$ and then calculated above is $\mathcal{L}^2(\bigcup_{i\in B_k} {Q}^k_{i}) < \delta_k$, otherwise $i\notin B_k$ and in this case $a_i \notin A_k$ and $I_{m_k}(a_i) < \frac{\epsilon_k}{4}$. From $I_{m_k}(a_i)< \frac{\epsilon_k}{4}$ we easily get the inequalities in $vi)$.
		
		It holds that 
		$$
			\bigcup_{i=1}^{O_k}\tilde Q_i^k = \Omega_k \setminus \Omega_{k-1}\,,\qquad \bigcup_{i=1}^{O_k}Q_i^k = v_k+(\Omega_k \setminus \Omega_{k-1})\,.
		$$ 
		As $v_k \in Q(0, 2^{-m_k-1})$, we have $\|v_k\|_\infty\le 2^{-m_k-1}$ and only squares $Q_i^k$ in the most external layer may intersect $\partial (\Omega_k\setminus \Omega_{k-1})$. Now we exclude some of the outer squares to leave a small gap to $\partial (\Omega_k \setminus \Omega_{k-1})$. If necessary we change the indexing to get the set $\{Q_i^k; 1\leq i\leq \tilde{O}_k\}$ and none of the squares $Q_i^k$ intersects $\partial (\Omega_k\setminus\Omega_{k-1})$ for $1\leq i\leq \tilde{O}_k$. Moreover, (because $\diam_{\infty}Q_i^k = 2^{1-m_k}$) for arbitrary $x$ in the outer squares of $\{Q_i^k; 1\leq i\leq \tilde{O}_k\}$ we have that
		$$
		0 \leq \dist_{\infty}(x, \partial (\Omega_k\setminus \Omega_{k-1})) <2\cdot 2^{-m_k}.
		$$
		We exclude also this layer of outer squares. Now for the elements of the outer squares of the remaining system $\{Q_i^k; 1\leq i\leq \tilde{\tilde{O}}_k\}$ one has
		$$
		2\cdot 2^{-m_k} \leq \dist_{\infty}(x, \partial (\Omega_k\setminus \Omega_{k-1})) <4\cdot 2^{-m_k}.
		$$
		We repeat this excluding process two more times and we get the final set $\{Q_i^k; 1\leq i\leq K_k\}$ with
		$$
		6\cdot 2^{-m_k} \leq \dist_{\infty}(x, \partial (\Omega_k\setminus \Omega_{k-1}))<8\cdot 2^{-m_k}
		$$
		and so satisfies $v)$. There is only a corner of $\partial \bigcup_i^{K_k}Q_i^k$ near a corner of $\partial (\Omega_k\setminus \Omega_{k-1})$ because $m_k \geq l_k+8$ (which is $iv)$). Simultaneously, because the corners of $\partial \bigcup_i^{K_k}Q_i^k$ have distance greater than $32\cdot 2^{1-m_k}$, any single square of $\{Q_i^k\}$ contains at most one corner of $\partial \bigcup_i^{K_k}Q_i^k$, therefore no square has more than two external sides and so shares at least two sides with other squares of $\{Q_i^k\}$, which is $iii)$.

		It is obvious that $i)$ holds; the union of the squares in $U_k$ has no holes, the shifted squares are the same and no holes inside can be formed by removing squares from the edge. The set $\bigcup_{i=1}^{K_k}Q_i^k$ consists of all of the squares.
		
		The last point to check is $ii)$. By the choice of $m_k \geq m''$ we can estimate
		$$
		\mathcal{L}^2\Big(\Omega_k\setminus \Big[\Omega_{k-1}\cup \bigcup_{i=1}^{K_k} {Q}_i^k\Big]\Big) < \delta_k
		$$
		by \eqref{IncyWincy} and so we get $ii)$.
	\end{proof}

	The next lemma is needed to describe how to slightly move the vertices of a grid to obtain better boundary values. This result is proved in \cite{C}, but we give the proof for the reader's convenience.

	\begin{lemma}\label{GridLock}
		Let $f\in W^{1,1}_{\text{loc}}(\Omega; \mathbb R^2)$ be a homeomorphism and $k, \varepsilon_k, \delta_k$ be given numbers. Let $\left\{Q^k_i=Q^k_i(a_i ,2^{m_k})\right\}$ be a grid of squares in $\Omega_{k}\setminus \Omega_{k-1}$ determined by given numbers $\varepsilon_k, \delta_k$ in Lemma~\ref{Segregation}. Then there exists  $\left\{\Q_i^k\right\}$  a grid of quadrilaterals verifying $\frac{3}{4}Q^k_i \subset \Q_i^k \subset \frac{5}{4}Q^k_i$ and such that $f$ is absolutely continuous on each side of $\partial \Q^k_i$ and for all $i$ it holds that 
		\begin{equation}\label{1507}
		\oint_{\partial \Q_i^k} |D_{\tau} f| \dx\mathcal{H}^1\leq C\oint_{2  Q_i^k} |Df| \dx \mathcal{L}^2.
		\end{equation}
		Further for those $i\notin B_k$ it holds that
		\begin{equation}\label{15071}
		\oint_{\partial \Q_i^k} |D_{\tau} f- Df(a_i) \tau| \dx\mathcal{H}^1 \leq C\varepsilon_k.
		\end{equation}
		In both cases above, $C>0$ is an absolute constant.
		
	\end{lemma}
	\begin{proof}
		For each vertex $V= (v_1, v_2)$ of the grid, we let $S_V$ be the segment 
		$$
		S_V=\left\{ (x, y): x\in\left[v_1- \frac{r_k}{8}, v_1+ \frac{r_k}{8}\right], y-v_2= x- v_1 \right\}.
		$$
		For each set of two neighbouring vertices $V_1$ and $V_2$ of the grid (i.e., $V_1$ and $V_2$ are endpoints of the same side of a square $Q$ of the grid), the following estimates hold
		\begin{equation}\label{starnov}
		\int_{S_{V_1}\times S_{V_2}}  \left(\int_{[X_1, X_2]} |D_{\tau} f(t)|\dx\mathcal{H}^1(t)   \right)\dx(\mathcal{H}^1\times \mathcal{H}^1)(X_1, X_2)  \leq C r_k \int_{2Q} |Df|\dx\mathcal{L}^2,
		\end{equation}
		\begin{equation}\label{star2}
		\begin{split}
		\int_{S_{V_1}\times S_{V_2}}  \left(\int_{[X_1, X_2]} |D_{\tau} f(t)- Df(a)\tau|\dx\mathcal{H}^1(t)   \right)&\dx(\mathcal{H}^1\times \mathcal{H}^1)(X_1, X_2)  \\  \leq{} & C r_k\int_{2Q} |Df- Df(a)|\dx\mathcal{L}^2.
		\end{split}
		\end{equation}
		
		Above, for $X_i \in S_{V_i}$, $i=1, 2$, $[X_1, X_2]$ denotes the segment whose endpoints are $X_1$ and $X_2$. The integrals in left-hand side are meaningful, as a consequence of the well-known property of the Sobolev mapping $f$ of being a.c. on almost all lines. Moreover, inequalities $\eqref{starnov}$ and  $\eqref{star2}$ hold  because $\mathcal H^1(S_{V_1})\approx r_k$ and $\text{co}(S_{V_1}\cup S_{V_2})\subset 2 Q$. Furthermore, $Q$ is any square of the grid such that $V_1, V_2\in \partial Q$.
		
		For simplicity, we only estimate \eqref{1507}. It will be clear that we shall be able to guarantee also \eqref{15071} for the `good' quadrilaterals.
		
		Fix $\lambda >4$. By Chebyshev's inequality, from \eqref{starnov} we deduce that there exist a subset $S(V_1, V_2)$ of $S_{V_1}$ and  a subset $S(V_2, V_1)$ of $S_{V_2}$ such that 
		$$\mathcal H^1 \left(S(V_1, V_2)\right)\geq \left(1- \frac{1}{\lambda}\right) \mathcal H^1 (S_{V_1}),$$
		$$\mathcal H^1 \left(S(V_2, V_1)\right)\geq \left(1- \frac{1}{\lambda}\right) \mathcal H^1 (S_{V_2}) $$
		and for every $X_1\in S(V_1, V_2)$ and $X_2\in S(V_2, V_1)$ the following estimate holds
		$$
		\int_{[X_1, X_2]} |D_{\tau} f(t)|\dx\mathcal H^1(t)  \leq \frac{\lambda^2}{r_k} C\int_{2Q} |Df|\dx \mathcal{L}^2.
		$$
		Now, let $V$ be a vertex of the grid. There are (at most) four vertices $V_1, V_2, V_3, V_4$ of the grid which are neighbouring vertices of  $V$ and therefore, by the above construction we find four subsets $S(V, V_i)$, $i=1,2,3,4$ of $S_V$ such that
		$$
		\mathcal S:= \bigcap_{i=1}^4 S(V, V_i)
		$$
		has positive $\mathcal H^1$-measure, since $\lambda >4$ and thus it is not empty. We replace $V$ with a $\tilde V\in \mathcal S$. To conclude, each square $Q= \text{co}\left\{V_1,V_2, V_3,V_4\right\}$ 
		of the grid will be replaced by the quadrilateral $\Q= \text{co}\left\{\tilde{V}_1,\tilde{V}_2, \tilde{V}_3, \tilde{V}_4\right\}$.
		
	\end{proof}

	For the following lemma recall that we assume \eqref{JPAE}.
	\begin{lemma}\label{PWLF}
		Let $\Omega \subset \er^2$ be a finitely connected bounded domain and let $f \in W^{1}X(\Omega, \er^{2})$ be a homeomorphism. Let $k\in \en$ and $\epsilon_k, \delta_k >0$ be given numbers. Let $\Omega_k$ be the set chosen in Lemma~\ref{Segregation} containing squares of type $Q_i^k = Q(a_i, 2^{-m_k})$ determined by the given numbers $\epsilon_k$,  $\delta_k$. Let $\Q_i^k$ be the quadrilaterals in $\Omega_k \setminus \Omega_{k-1}$ derived from $Q_i^k$ by Lemma~\ref{GridLock}. Then there exists a piece-wise linear injective function $h$ defined on $\Gamma_k = \bigcup_i \partial\Q_i^k$ such that:
		\begin{enumerate}
			\item[$i)$] it holds that
			$$
			\|f-h\|_{L^{\infty}(\Gamma_k)}  <4\epsilon_k,
			$$
			\item[$ii)$] at each vertex $x$ of each $\Q_i^k$ it holds that $h(x) = f(x)$,
			\item[$iii)$] for any $S$ a side of a $\Q_i^k$ and for $\H^1$ a.e. $x \in S$ we have
			\begin{equation}\label{SpiderMan}
			|D_{\tau}h(x)| \leq \oint_S |D_{\tau}f|\dx \H^1,
			\end{equation}
			\item[$iv)$] for any $i \notin B_k$ it holds that
			$$
			\oint_{\partial \Q_i^k} |Df(a_i)\tau - D_{\tau}h| \dx\H^1 \leq C\,\epsilon_k.\ 
			$$
			\item[$v)$] for any $i \notin B_k$ with $\sqrt{\epsilon_k}<|Df(a_i)|<\tfrac{1}{\sqrt{\epsilon_k}}$ and $J_f(a_i)>4\sqrt{\epsilon_k}$ it holds that $h$ is linear on each $S$ side of $\Q_i^k$.
		\end{enumerate}
	\end{lemma}
	\begin{proof}
		The first step is to define a piece-wise linear approximation of $f$, (call it $\tilde{f}_n$) on each side $S$ of every $\Q_i^k$. We call $x_1,\dots, x_n$ $n$ evenly spaced points along $S$ with $x_1$ and $x_n$ being the two endpoints of $S$. Then we put $\tilde{f}_n(x_j) = f(x_j)$ and $\tilde{f}_n$ is linear on each segment between $x_j$ and $x_{j+1}$. Of course $\tilde{f}_n$ converges uniformly to $f$ on $S$ and for $n$ large enough is injective. The argument is rather simple and a detailed version can be found in \cite{HP}, \cite{C} or an alternative approach can be found in \cite{DPP}. We define $h$ so that $h(S) = \tilde{f}(S)$ for each $S$, but $h$ parametrizes its image from $S$ at constant speed.
		
		Point $i)$ holds because of Lemma~\ref{Segregation} point $vii)$ and because the oscillation of $h$ on any $\partial \Q_i^k$ is bounded by the oscillation of $f$ on the $\partial \Q_i^k$ which is bounded by the oscillation of $f$ on the $2Q_i^k$.
		
		Point $ii)$ is obvious. Point $iii)$ holds because the piece-wise linear curve is not longer than the original curve i.e.
		$$
		\oint_S|D_{\tau}h| \dx\H^1 \leq \oint_S |D_{\tau}f| \dx\H^1,
		$$
		and because the derivative has constant size, i.e.
		$$
		|D_{\tau}h(x)| = \oint_S |D_{\tau}h| \dx\H^1
		$$
		almost everywhere.
		
		Let $x_{m}, x_{m+1}$ be a pair of adjacent points on $S$ such that the derivative of $h$ is constant on the segment $L = [x_mx_{m+1}]$ and $h(x_m) = f(x_m)$ and $h(x_{m+1}) = f(x_{m+1})$. Then we can calculate for all $t\in L$ that 
		$$
		D_{\tau}h(t) = \oint_LD_{\tau} f\dx\H^1.
		$$
		When we denote $L(t)$ as the segment $L\subset \partial\Q_i^k$ described above that contains the point $t\in\partial\Q_i^k$ we can easily calculate
		$$
		\begin{aligned}
		\oint_{\partial\Q_i^k}\big|\oint_{L(t)} D_{\tau}f\dx \H^1 - Df(a_i)\tau \big|\dx \H^{1}(t) &\leq \oint_{\partial\Q_i^k}\big|\oint_{L(t)} D_{\tau}f - Df(a_i)\tau \dx\H^1\big|\dx \H^{1}(t) \\
		&\leq \oint_{\partial\Q_i^k}\oint_{L(t)} |D_{\tau}f - Df(a_i)\tau| \dx \H^1\dx \H^{1}(t) \\
		&\leq  \tfrac{1}{\H^1(\partial\Q_i^k)} \sum_{L\subset\partial\Q_i^k} \H^1(L)\oint_{L} |D_{\tau}f - Df(a_i)\tau| \dx \H^1 \\
		&\leq  \oint_{\partial\Q_i^k}|D_{\tau}f - Df(a_i)\tau| \dx \H^1.
		\end{aligned}
		$$
		Now for any $i\notin B_k$ we can see by \eqref{15071} of Lemma~\ref{GridLock} that the above can further be estimated by $C\,\epsilon_k$, which is exactly point $iv)$.
		
		We now prove point $v)$. At any point $x$ where $J_f(x) > 0$ it holds that $J_f(x) = \lambda_1\lambda_2$, where $\lambda_1 = \max\{|Df(x)v|; |v|=1\}$ and $\lambda_2 = \min\{|Df(x)v|; |v|=1\}$. Given that $\lambda_1 < \tfrac{1}{\sqrt{\epsilon_k}}$ and $J_f(x) > 4 \sqrt{\epsilon_k}$ it follows that $\lambda_2 > 4\epsilon_k$ (in the following we will use this fact repeatedly to estimate distance between images). In Lemma~\ref{GridLock} we made $\Q_i^k$ from $Q_i^k$ by moving its vertexes along a southwest-northeast diagonal (we will call it now the SW-NE diagonal). Therefore the distance of the northwest (we will call it now the NW) and southeast (SE) vertexes to the SW-NE diagonal is not changed in the process. This distance is $\sqrt{2}\cdot 2^{-m_k}$. Therefore the distance of the images in $Df(a_i)$ of the NW and SE vertexes of $\Q_i^k$ to the image of the SW-NE diagonal is at least $4\sqrt{2}\cdot 2^{-m_k}\epsilon_k$. This means that the balls of radius $2^{-m_k}\epsilon_k$ centered at the image of the NW and SE vertexes of $\Q_i^k$ both lie entirely on different sides of the line containing the image of the SW-NE diagonal. In fact they have
		\begin{equation}\label{Compass}
		\text{ distance from the image of the diagonal of more than $2^{-m_k}\epsilon_k$.}
		\end{equation}
		The SW and NE vertexes of $\Q_i^k$ lie on a SW-NE line through $a_i$ and so the images of $a_i$ and the SW and NE vertices lie on the line containing the image of the SW-NE diagonal. It holds that $\Q_i^k \supset \tfrac{3}{4}Q_i^k$ and so the distance of the SW (NE) vertex from $a_i$ is at least $\tfrac{3\sqrt{2}}{4}2^{-m_k}$ (and the images of the vertexes lie on opposite direction from sides of the image of $a_i$). Thus the distance of the image of the SW (NE) vertex in $Df(a_i)$ to the image of $a_i$ in $Df(a_i)$ is at least $3\sqrt{2}\cdot2^{-m_k}\epsilon_k$. This means that the ball of radius $2^{-m_k}\epsilon_k$ centered at the image of the SW (or NE) vertex lies entirely on different sides of the line containing the point $Df(a_i)a_i$ and perpendicular to the image of the SW-NE diagonal. Therefore, and thanks to \eqref{Compass}, any pair of balls of radius $2^{-m_k}\epsilon_k$ centered at the image of a pair of vertexes do not intersect. Now any triangle having a vertex in each of the 3 balls of radius $2^{-m_k}\epsilon_k$ centered at the image in $Df(a_i)$ of a the SW, NW and NE vertex (or similarly for SW, NE, SW vertex) of $\Q_i^k$ is a positively oriented Jordan curve (there is a simple homotopy with $Df(a_i)$ on SW NW NE). Because the image of the vertexes of $\Q_i^k$ in $f$ each lie in a ball of radius $2^{-m_k}\epsilon_k$ around the image in $Df(a_i)$ of the same vertex the 2-piece-wise affine map (divide $\Q_i^k$ by the SW-NE diagonal) coinciding with $f$ on the vertices of $\Q_i^k$ has positive orientation and is injective on $\Q_i^k$.
	\end{proof}
	
	\begin{lemma}\label{Click}
		Let $\Omega \subset \er^2$ be a finitely connected bounded domain and let $f \in W^{1}X(\Omega, \er^{2})$ be a homeomorphism. Let $k\in \en_0$ and $\epsilon_k, \delta_k >0$
		. Let $\Omega_k \subset \Omega_{k+1}$ be a pair of sets chosen in Lemma~\ref{Segregation} which we apply with the given $\epsilon_k$, $\epsilon_{k+1}$ and $\delta_k$, $\delta_{k+1}$ giving squares $\{Q_i^k\}_i$, resp $\{Q_i^{k+1}\}_{i}$. Let $\Q_i^k$ be the quadrilaterals from Lemma~\ref{GridLock} in $\Omega_k \setminus \Omega_{k-1}$ and $\Q_i^{k+1}$ be the quadrilaterals from Lemma~\ref{GridLock} in $\Omega_{k+1}\setminus \Omega_{k}$. Call $S_k$ the space between $\Q_i^{k}$ and $\Q_i^{k+1}$, i.e. the union of all components of $\Omega \setminus [\bigcup_i \Q_i^k \cup \bigcup_i \Q_i^{k+1}]$ which intersects $\partial \Omega_k$. Finally let $h$ be the piece-wise linear function determined by Lemma~\ref{PWLF}. Then there exists a finite collection of quadrilaterals $\{{\Q}_i^{k}\}_{i=K_k+1}^{N_k}$ such that $\bigcup_{i=K_k+1}^{N_k} {\Q}_i^{k} = S_k$ and a finitely piece-wise affine homeomorphism $g$ defined on $S_k$ such that $g = h$ on $\partial S_k$ and
		\begin{equation}\label{Halelujah}
		\|Dg\chi_{S_k}\|_{X(\Omega)} \le C\|Df\chi_{S_k}\|_{X(\Omega)}.
		\end{equation}
	\end{lemma}
	\begin{proof}
		
		We start by assuming that $\Omega$ is simply connected and $\Omega_{k+1}\setminus \Omega_k$ has exactly one component. If this were not so, then we could deal with each component seperately in the same way.

		\begin{figure}
			\centering
			
			\includegraphics[width=11cm]{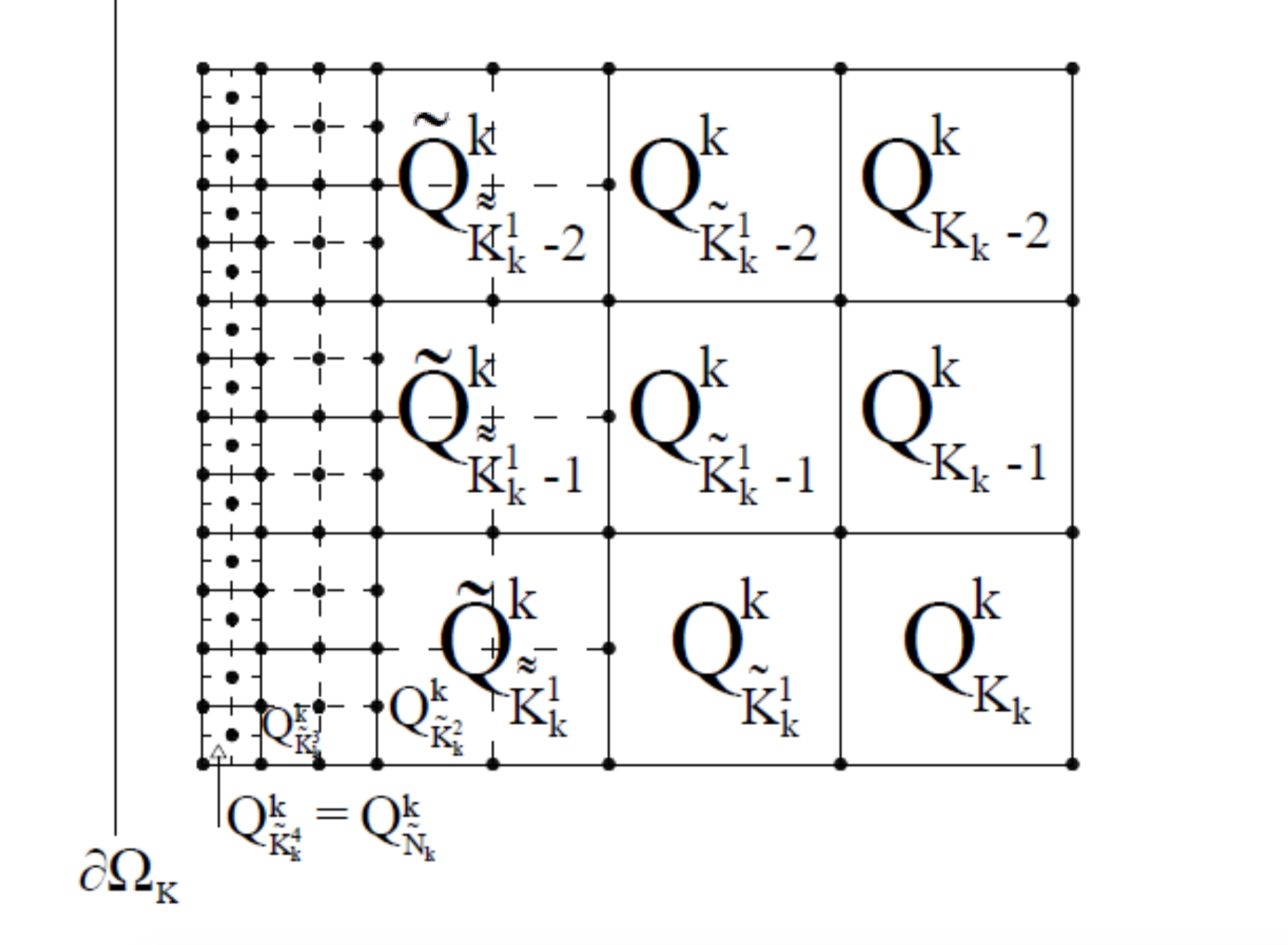}
			\caption{The gradual halving of squares untill we have squares with the same sidelength as in the neighbouring set}\label{fig:Fill1}
		\end{figure}

		\step{1}{Go from squares of size $2^{-m_k}$ to squares of size $2^{-m_{k+1}}$}{A}
		Call $K_k^1 = K_k$. We have the quadrilaterals $\Q_{i}^k$, $1\le i\le K_k^1$, from Lemma~\ref{GridLock} which were constructed from the squares $Q_i^k\in \mathcal{D}_{-m_k}^{v_k}$ from Lemma~\ref{Segregation}. Now we name the neighbours of $Q_i^k$. We index the finite set of neighbouring squares
		$$
		\Big\{Q \in \mathcal{D}_{-m_k}^{v_k}; Q^{\circ}\cap \Big(\bigcup_{i=1}^{{K}_k^1}Q_i^k\Big) = \emptyset, \overline{Q} \cap\overline{\Big(\bigcup_{i=1}^{{K}_k^1}Q_i^k\Big)} \neq \emptyset\Big\} =\{{Q}_i^k\}_{ i =K_k+1}^{\tilde{K}_k^1}.
		$$
		That is to say the squares ${Q}_i^k$, $K_k^1+1\leq i \leq \tilde{K}_k^1$ are those shifted dyadic squares that are not contained in the set $\bigcup_{i=1}^{K_k^1}Q_i^k$ but share at least one common vertex with them (see Figure~\ref{fig:Fill1}). By $v)$ of Lemma~\ref{Segregation} and by $\diam_{\infty} Q(x,2^{-m_k}) = 2^{1-m_k}$ we have that
		$$
		4\cdot 2^{-m_k} \leq \dist_{\infty}\bigg(x, \partial \Big[\bigcup_{i=1}^{\tilde{K}_k^1}{Q}_i^k\Big]\bigg)
		$$
		for all $x\in \partial \Omega_k$. Similarly call the neighbours of the previously added squares 
		$$
		\Big\{Q \in \mathcal{D}_{-m_k}^{v_k}; Q^{\circ}\cap \Big(\bigcup_{i=1}^{\tilde{K}_k^1}Q_i^k\Big) = \emptyset, \overline{Q}\cap\overline{\Big(\bigcup_{i=1}^{\tilde{K}_k^1}Q_i^k\Big) } \neq \emptyset\Big\} =\{\tilde{Q}_i^k\}_{ i =\tilde{K}_k^1+1}^{\tilde{\tilde{K}}_k^1}.
		$$
		Then
		$$
		2\cdot 2^{-m_k} \leq \dist_{\infty}\bigg(x, \partial\Big[ \bigcup_{i=1}^{{\tilde{K}}_k^1}{Q}_i^k\cup \bigcup_{i={\tilde{K}}_k^1+1}^{\tilde{\tilde{K}}_k^1}\tilde{Q}_i^k\Big]\bigg).
		$$
		for all $x\in \partial \Omega_k$.
		
		We divide each square $\tilde{Q}_i^k = Q(\tilde{a}_i, 2^{-m_k})$ for $\tilde{K}_k^1+1 \leq i \leq \tilde{\tilde{K}}_k^1$ into its four quarters $Q(\tilde{a}_i + 2^{-m_k-1}z, 2^{-m_k-1})$, for each $z$ a corner of $Q(0,1)$. We number the squares we get by dividing $\tilde{Q}_i^k$, $\tilde{K}_k^1+1 \leq i \leq \tilde{\tilde{K}}_k^1$ into quarters as $Q_i^k = Q(a_i, 2^{-m_k-1})$ for $\tilde{K}_k^1+1 \leq i \leq {\tilde{K}}_k^2$. Then we have
		$$
		4\cdot 2^{-m_k-1} \leq \dist_{\infty}\bigg(x, \partial \Big[\bigcup_{i=1}^{\tilde{K}_k^2}{Q}_i^k\Big]\bigg)
		$$
		for all $x\in \partial \Omega_k$.
		
		We now repeat this last operation. We call the neighbours of the previously added squares
		$$
		\Big\{Q \in \mathcal{D}_{-m_k-1}^{v_k}; Q^{\circ}\cap  \Big( \bigcup_{i=1}^{\tilde{K}_k^2}Q_i^k\Big) = \emptyset, \overline{Q}\cap\overline{\Big(\bigcup_{i=1}^{\tilde{K}_k^2}Q_i^k\Big)} \neq \emptyset\Big\} =\{\tilde{Q}_i^k\}_{ i =\tilde{K}_k^2+1}^{\tilde{\tilde{K}}_k^2}.
		$$
		We divide the squares $\{\tilde{Q}_i^k = Q(\tilde{a}_i, 2^{-m_k-1})\}_{i = \tilde{K}_k^2+1}^{\tilde{\tilde{K}}_k^2}$ into its four quarters $Q(\tilde{a}_i + 2^{-m_k-2}z, 2^{-m_k-2}) \in \mathcal{D}_{-m_k-2}^{v_k}$, for each $z$ a corner of $Q(0,1)$ and call these squares $Q_i^k = Q(a_i, 2^{-m_k-2})$, $\tilde{K}_k^2+1 \leq i \leq \tilde{K}_k^3$. As before we have that
		$$
		4\cdot 2^{-m_k-2} \leq \dist_{\infty}\Big(x, \partial \bigcup_{i=1}^{\tilde{K}_k^3}{Q}_i^k\Big)
		$$
		for all $x\in \partial \Omega_k$.
		
		After $m_{k+1} - m_k+1$ steps we start adding squares of type $Q(a_i, 2^{-m_{k+1}})$. We call those $Q_i^k$, $\tilde{K}_k^{m_{k+1} - m_k+1} \leq i \leq \tilde{N}_k$ the dyadic squares of type $Q(a_i, 2^{-m_{k+1}}) \subset \Omega_k$ which are not in $\bigcup_{i=1}^{\tilde{K}_k^{m_{k+1} - m_k+1}}Q_i^k$. Therefore
		\begin{equation}\label{Sandwich}
		\text{the distance from $\partial \bigcup_{i=1}^{\tilde{N}_k}Q_i^k$ to $\partial \Omega_k$ is between 2 and 4 times $2^{-m_{k+1}}$.}
		\end{equation}
		By adding shifted dyadic squares of type $Q\in \mathcal{D}_{-m_{k+1}}^{v_{k+1}}$ to the grid in $\Omega_{k+1} \setminus \Omega_k$ we can guarantee that the distance from $\partial \bigcup_i Q_i^{k+1}$ to $\partial \Omega_k$ is also between $2$ and $4$ times $2^{-m_{k+1}}$. Thus the distance between $\partial \bigcup_{i=1}^{\tilde{N}_{k}} Q_i^k$ and $\partial \bigcup_{i=1}^{\tilde{N}_{k+1}} Q_i^{k+1}$ is between $4$ and $8$ times $2^{-m_{k+1}}$. Call $\tilde{G}_k$ the component of $\er^2 \setminus (\bigcup_{i=1}^{\tilde{N}_{k}} Q_i^k \cup \bigcup_{i=1}^{\tilde{N}_{k+1}} Q_i^{k+1})$ containing $\partial \Omega_k$.
		
		\begin{figure}
			\centering
			\includegraphics[width=11cm]{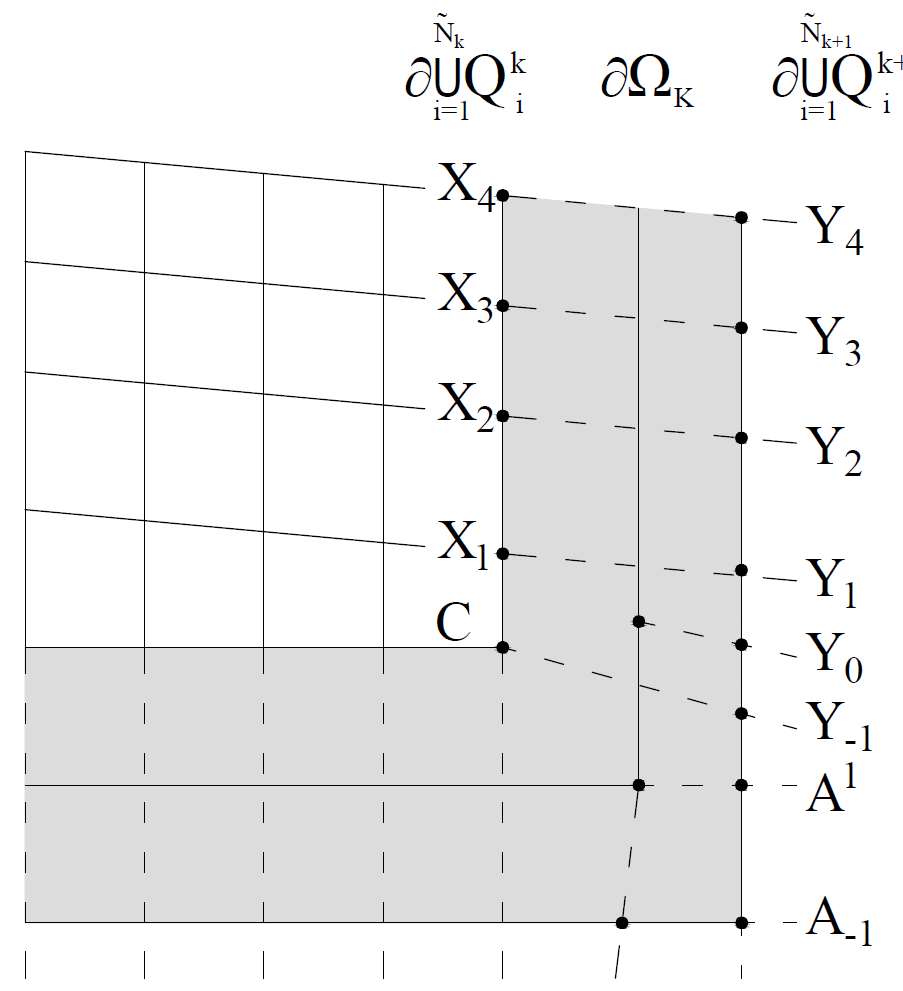}
	
			\caption{Filling in the left over space (in the picture shaded) between two neighbouring grids with quadrilaterals uniformly piece-wise affine bi-Lipschitz equivalent with $Q(0,2^{-m_{k+1}})$. The line vertices $X_i$ are partnered with the line vertices $Y_i$. The vertex $A_{-1}$ is a reflex vertex and $C$ is an acute vertex. We the two vertices adjacent to $A_{-1}$ (one of them is denoted as $A^1$ the other is not denoted) to the near vertex of $\partial \Omega_k$ with a segment. The vertex $C$ is connected with the next available neighbour of $A$ (in this case $Y_{-1}$). The vertex $Y_0$ is connected with $\partial \Omega_k$.}\label{fig:Fill2}
		\end{figure}
		
		In our construction we always have two layers of squares of side length $2^{1-m_k-j}$ before we start adding squares of side length $2^{1-m_k-(j+1)}$. Therefore
		\begin{equation}\label{BoundMe}
			\begin{aligned}
				& 2Q_i^k \text{ intersects at most its neighbours and its neighbours' neighbours but no} \\
				&\text{other squares further away.} 
			\end{aligned}
		\end{equation}
		Further the square $2Q_i^k$ is contained in the neighbours of $Q_i^k$. Therefore any square $Q_{i'}^k$ such that $2Q_{i'}^k$ intersects $2Q_i^k$ then it is at most a third neighbour (i.e. a neighbour of $Q_i^k$ a neighbour's neighbour or a neighbour of a neighbour of a neighbour of $Q_i^k$). But the number of third neighbours is bounded. Therefore there exists an $L$ such that
		\begin{equation}\label{AlmostFiniteOverlaps1}
		\sum_{i=1}^{\tilde{N}_k}\chi_{2Q_i^k}(x)\leq L
		\end{equation}
		for almost every $x \in S_k$.

		\step{2}{Fill the rest of $S_k$ with quadrilaterals}{B}
		What is left over is a `tube' around $\partial\Omega_k$ approximately $2^{-m_{k+1}}$ wide, which can be divided into quadrilaterals all uniformly bi-Lipschitz equivalent with a square of side length $2^{1-m_{k+1}}$. Although this fact is obvious we describe one way how to do this in detail here.
		
		We call those vertices $X$ of squares $Q_i^k$ such that $X \in \partial \bigcup_{i=1}^{\tilde{N}_k} Q_i^k$ outer vertices. Similarly we also call those vertices $X$ of squares $Q_i^{k+1}$ such that $X \in \partial \bigcup_{i=1}^{\tilde{N}_{k+1}} Q_i^{k+1}$ outer vertices. From $iv)$ of Lemma~\ref{Segregation} we have that $\partial \bigcup_{i=1}^{\tilde{N}_k} Q_i^{k}$ and $\partial \bigcup_{i=1}^{\tilde{N}_{k+1}} Q_i^{k+1}$ are piece-wise linear parallel to coordinate axes and each of their sides corresponds to a side of $\partial \Omega_k$. We call a vertex of $\partial \bigcup_{i=1}^{\tilde{N}_{k}} Q_i^{k}$ (or $\partial \bigcup_{i=1}^{\tilde{N}_{k+1}} Q_i^{k+1}$) a reflex corner of $\partial \bigcup_{i=1}^{\tilde{N}_{k}} Q_i^{k}$ (or $\partial \bigcup_{i=1}^{\tilde{N}_{k+1}} Q_i^{k+1}$) if $270^{\circ}$ of a small circular arc centered at the vertex lies inside $\bigcup_{i=1}^{\tilde{N}_{k}} Q_i^{k}$ (or $\bigcup_{i=1}^{\tilde{N}_{k+1}} Q_i^{k+1}$) for example $A_{-1}$ in Figure~\ref{fig:Fill2}. In the other case we call a vertex of $\partial \bigcup_{i=1}^{\tilde{N}_{k}} Q_i^{k}$ (or $\partial \bigcup_{i=1}^{\tilde{N}_{k+1}} Q_i^{k+1}$) an acute corner of $\partial \bigcup_{i=1}^{\tilde{N}_{k}} Q_i^{k}$ (or $\partial \bigcup_{i=1}^{\tilde{N}_{k+1}} Q_i^{k+1}$) if $90^{\circ}$ of a small circular arc centered at the vertex lies inside $\bigcup_{i=1}^{\tilde{N}_{k}} Q_i^{k}$ (or $\bigcup_{i=1}^{\tilde{N}_{k+1}} Q_i^{k+1}$) for example $C$ in Figure~\ref{fig:Fill2}. We describe an outer vertex as a line vertex if it is not a corner of $\partial \bigcup_{i=1}^{\tilde{N}_{k}} Q_i^{k}$ neither is it the vertex neighbouring a relfex corner of $\partial \bigcup_{i=1}^{\tilde{N}_{k}} Q_i^{k}$ (similarly for $\bigcup_{i=1}^{\tilde{N}_{k+1}} Q_i^{k+1}$) for example $X_i$ or $Y_i$ in Figure~\ref{fig:Fill2}.
		
		Firstly we deal with the area around corners of $\partial \Omega_k$. We add no new segments going from a reflex corner $A$. We add a segment from the vertices $A', A''$ neighbouring $A$ to the nearby corner of $\partial\Omega_k$, this creates a quadrilateral (see Figure~\ref{fig:Fill2}). For an acute corner $C$ with adjacent sides $S_1$ and ${S}_2$ (whose corresponding opposite sides are $\tilde{S}_1$ and $\tilde{S}_2$) we create new quadrilaterals by adding a segment from $C$ to both of the second neighbours of $A$.
		
		Now choose any side of $\partial \Omega_k$ and consider the two corresponding sides $S, \tilde{S}$, with $S$ a side of $\partial \bigcup_{i=1}^{\tilde{N}_k} Q_i^{k}$ and $\tilde{S}$ a side of $\partial \bigcup_{i=1}^{\tilde{N}_{k+1}} Q_i^{k+1}$. Both $S$ and $\tilde{S}$ have endpoints very close to the same corners of $\partial\Omega_k$, see \eqref{Sandwich}. This means that in the extreme scenario the distance between the length of $S$ and $\tilde{S}$ is at less than $8\cdot 2^{-m_{k+1}}$. Because the side length of the squares is $2\cdot 2^{-m_{k+1}}$ the difference in number of line vertices on $S$ and $\tilde{S}$ varies by at most 4, at most 2 at each corner. Thus it is possible to pair the line vertices of $S$ and the line vertices of $\tilde{S}$ so that at most two line vertices of $S$ (resp. $\tilde{S}$) close to a given corner of $S$, (resp. $\tilde{S}$) do not have a partner in $\tilde{S}$, (resp $S$) moreover the segment between a pair of partnered vertexes is nearly perpendicular to the corresponding segment of $\partial \Omega_k$. Let $X_1, X_2$ be a pair of line vertices on $S$ with partners $Y_1$ and $Y_2$, their partners in $\tilde{S}$. Each of the segments $X_1Y_1$ and $X_2Y_2$ intersect $\partial\Omega_k$ exactly once, call the points $Z_1$ and $Z_2$ respectively. Any of the quadrilaterals $X_1Z_1Z_2X_2$ and  $Y_1Z_1Z_2Y_2$ formed by any two pairs of neighbouring line vertexes are uniformly bi-Lipschitz equivalent with $Q(0,2^{-m_{k+1}})$ by a 2-piece-wise affine map. Permitting a small bastardisation of the notation we call these quadrilaterals $Q_i^k$ for $\tilde{N}_k+1\leq i\leq \tilde{\tilde{N}}_k$.
		
		If we have a left over line vertex $X$ with no partner we create a new quadrilateral with a vertex at $X$ by adding a segment from $X$ to the point on $\partial \Omega_k$ half way between the neighbouring vertices on $\partial \Omega_k$ (see $Y_0$ in Figure~\ref{fig:Fill2}). We call the entire collection of these quadrilaterals $Q_i^k$, $1\leq i\leq N_k$. Up to maybe increase $L$ the equation \eqref{AlmostFiniteOverlaps1} is extended as
		\begin{equation}\label{AlmostFiniteOverlaps}
		\sum_{i=1}^{{N}_k}\chi_{2Q_i^k}(x)\leq L
		\end{equation}
		
		\step{3}{Move the corners}{C}
		It is possible by moving the corners of the quadrilaterals of  $Q_i^k$ $K_k\leq i \leq N_k$ to get quadrilaterals $\Q_i^k$ which satisfy the estimate \eqref{1507}. The idea is exactly that of the proof of Lemma~\ref{GridLock}, the only difference here is that neighbouring squares may not have exactly the same side length, but the ratio is bounded by 2. This case is dealt with in detail in \cite[Theorem 4.1, step 2]{C} and the interested reader can check the details there. We move corners by at most $2^{-2-m_k-j}$, for $Q_i^k$ in the $j$-th generation of added squares.
		
		\step{4}{Define $g$ on the grid and extend}{D}
		We define $g$ on $\partial \Q_i^k$ for $K_k+1 \leq i \leq N_k$ exactly the same way as we defined $h$ in Lemma~\ref{PWLF} for $1\leq i\leq K_k+1$. We get a piece-wise linear and injective function $g$ on $\bigcup_{i=K_k+1}^{N_k} \partial\Q_i^k$. Because each $\Q_i^k$ is uniformly bi-Lipschitz equivalent with a square by (uniformly bounded number of pieces) piece-wise affine maps we can apply Corollary~\ref{RePara} and get a finitely piece-wise affine homeomorphism $g$ on $S_k$, extending the original mapping and satisfying
		$$
		\|Dg\|_{L^{\infty}(\Q_i^k)} \leq \frac{C}{r} \int_{\partial\Q_i^k}|D_{\tau}f|\dx \H^1.
		$$
		Further by applying the \eqref{1507} type estimate we got from step 3, we get
		\begin{equation}\label{Filip}
		\|Dg\|_{L^{\infty}(\Q_i^k)} < C\oint_{2Q_i^k}|Df| \dx \mathcal{L}^2
		\end{equation}
		on each $Q_{i}^k$, for $K_k+1 \leq i \leq N_k$.
		
		Let $K_k+1 \leq i \leq N_k$, let $y \in \Q^k_i$ and let $r>0$. For the simplicity of notation call
		$$
			A_{k,y,r} = \{l\in \en:K_k+1 \leq l \leq N_k;\mathcal{L}^2(\Q_l^k \cap Q(y,r)) >0\}\,,
		$$
		then it is obvious that
		\begin{equation}\label{Estove}
			\oint_{Q(y,r)}|Dg \chi_{\bigcup_{i=K_k+1}^{N_k}\Q_i^k}| \dx \mathcal{L}^2 \leq \max_{l \in A_{k,y,r}} \| Dg\|_{L^\infty(\Q_l^k)}.
		\end{equation}
		But in step~\ref{C} we moved corners by at most $2^{-2-m_k-j}$ for squares in the $j$-th generation. Therefore, for our $K_k+1\leq i \leq N_k$ (say $Q_i^k  =Q(a_i,2^{-m_k-j})$ is a $j$-th generation square) we have that $\Q_i^k \subset Q(a_i,\tfrac{5}{4}2^{-m_k-j})$ and this implies that for each $y\in \Q_i^k$ and $0<r<2^{-2-m_k-j}$ it holds that $Q(y,r) \subset 2Q_i^k$. Therefore, for any $l \in A_{k,y,r}$ we have that $\Q_l^k \cap 2Q_i^k \neq \emptyset$. On the other hand any `square' intersecting $2Q_i^k$ must be contained in $6Q_i^k$, because neighbours (and second neighbours) have side lengths bounded by a factor of 2. But then for $l \in A_{k,y,r}$
		\begin{equation}\label{contained}
			2Q_l^k\subset 12Q_i^k.
		\end{equation}
		Further, thanks to \eqref{BoundMe} we have that
		\begin{equation}\label{BoundYou}
			\mathcal{L}^2(Q_l^k) \geq C\mathcal{L}^2(Q_i^k),
		\end{equation}
		with $C$ independent of $i$, $j$, $k$ and $l$. We continue \eqref{Estove} using \eqref{Filip} and then \eqref{BoundYou} together with \eqref{contained} and get
		\begin{equation}\label{Part1}
		\begin{aligned}
			\oint_{Q(y,r)}|Dg \chi_{\bigcup_{l=K_k+1}^{N_k}\Q_l^k}| \dx \mathcal{L}^2 &\leq \max_{l \in A_{k,y,r}} \| Dg\|_{L^\infty(\Q_l^k)}\\
			& \leq  C\,\max_{l \in A_{k,y,r}} \oint_{2Q_l^k}|Df| \dx \mathcal{L}^2 \\
			&\leq C\oint_{12Q_i^k}|Df| \dx \mathcal{L}^2 \leq C M(|Df|)(y)
		\end{aligned}
		\end{equation}
		where $M$ denotes the maximal operator.
		
		Now we consider a $y\in \Q_i^k$ and $r>2^{-2-m_k-j}$, then, using \eqref{hope} and \eqref{1507} we have
		$$
		\begin{aligned}
			\int_{Q(y,r)}|Dg\chi_{\bigcup_{l=K_k+1}^{N_k}\Q_l^k}|\dx \mathcal{L}^2  &= \sum_{l=K_k+1}^{N_k}\int_{Q(y,r)\cap\Q_l^k}|Dg|\dx \mathcal{L}^2 \\
			&\leq C\sum_{l \in A_{k,y,r}}\int_{2Q_l^k}|Df|\dx \mathcal{L}^2.
		\end{aligned}
		$$
		We use $\eqref{contained}$, the fact that $r>2^{-2-m_k-j}$ to get that $Q(y,52r) \supset \bigcup_{l \in A_{k,y,r}}2Q_l^k$. Then in the previous estimate, using \eqref{AlmostFiniteOverlaps}, we get
		$$
		\sum_{l \in A_{k,y,r}}\int_{2Q_l^k}|Df|
		\leq C L \int_{Q(y,52r)}|Df\chi_{\bigcup_{i=K_k+1}^{N_k}\Q_i^k}|
		$$
		for $r>2^{-2-m_k-j}$. Therefore
		\begin{equation}\label{Part2}
		\begin{aligned}
		\oint_{Q(y,r)}|Dg\chi_{\bigcup_{i=K_k+1}^{N_k}\Q_i^k}| \dx \mathcal{L}^2 
		&\leq C\oint_{Q(y,52r)}|Df\chi_{\bigcup_{i=K_k+1}^{N_k}\Q_i^k}|\dx \mathcal{L}^2 \\
		&\leq C M(|Df|\chi_{\bigcup_{i=K_k+1}^{N_k}\Q_i^k})(y).
		\end{aligned}
		\end{equation}
		By taking supremum over all $0 < r < \infty$ in \eqref{Part1} and \eqref{Part2} on the left-hand side we obtain that
		\begin{equation}\label{SomeStuff}
		M(|Dg|\chi_{\bigcup_{i=K_k+1}^{N_k}\Q_i^k})(x)\leq C M(|Df|\chi_{\bigcup_{i=K_k+1}^{N_k}\Q_i^k})(x).
		\end{equation}%
		By the norm comparison \eqref{norm-comparison} we obtain that 
		$$
		\left\|Dg\chi_{\bigcup_{i=K_k+1}^{N_k}\Q_i^k}\right\|_{X}\leq C	\left\| Df\chi_{\bigcup_{i=K_k+1}^{N_k}\Q_i^k}\right\|_{X}
		$$
		which is \eqref{Halelujah}.	
	\end{proof}
	
	\section{Proof of Theorem~\ref{main}\label{sect5}}
	
	We now prove Theorem~\ref{main}.
	
	\step {1}{A grid of `squares'}{AB}
	
	Let $\epsilon^* >0$ be any fixed positive number, we want to find a piece-wise affine homeomorphic approximation $\hat{f}$ such that
	$$
	\|Df-D\hat{f}\|_{X(\Omega)}< C\epsilon^*
	$$
	and
	$$
	\|f-\hat{f}\|_{L^{\infty}(\Omega)} < \epsilon^*.
	$$
	The first stage of this is to separate $\Omega$ into nested sets $\Omega_k$. This separation is the subject of Lemma~\ref{Segregation}. We want to apply Lemma~\ref{Segregation} and get a grid of squares $Q_i^k$ but in order to do this we need to choose $\epsilon_k, \delta_k>0$ and a set $A_k$. Using the `$\epsilon$-$\delta$' continuity of the norm of $X$ (recall that all spaces supporting the Lebesgue Point Property have this property) we choose $\delta_k$ be a number so small that
	\begin{equation}\label{ChooseDelta}
	\|Df\chi_{E}\|_{X(\Omega)}<2^{-k}\epsilon^*
	\text{ for any } E\subset \Omega_k \text{ such that }\mathcal{L}^2(E)< 5\delta_k.
	\end{equation}
	We find a number $T_k\geq 1$ such that
	\begin{equation}\label{StartWithT}
	\begin{aligned}
	\mathcal{L}^2\Big(\big\{x:|Df(x)| > T_k \big\}\Big) &< \frac{\delta_k}{96},\\
	\mathcal{L}^2\Big(\big\{x:0<|Df(x)| <T_k^{-1} \big\}\Big) &< \frac{\delta_k}{96}  \text{ and } \\
	\mathcal{L}^2\Big(\big\{x:0<|J_f(x)| < 4 T_k^{-1} \big\}\Big) &< \frac{\delta_k}{96}.\\
	\end{aligned}
	\end{equation}
	Then we call $A_k$ the union of these sets, i.e.
	$$
	A_k = \{x:|Df(x)| > T_k \big\}\cup\{x:0<|Df(x)| <T_k^{-1} \big\}\cup\{x:0<|J_f(x)| < 4 T_k^{-1} \big\}
	$$
	and $\mathcal{L}^2(A_k) < \tfrac{\delta_k}{32}$. Now we apply Lemma~\ref{Rozumny} with $G = \Omega_k \setminus \Omega_{k-1}$, $M = (2+C_{\eqref{hope}})T_k$, $\tilde{\epsilon}_k = \frac{2^{-k}\epsilon^*}{\mathcal{L}^2(\Omega_k\setminus \Omega_{k-1})}$. This gives us a number $\tilde{\delta}_k$. We require
	\begin{equation}\label{ChooseEpsilon}
	\begin{aligned}
	\epsilon_k&< \frac{\tilde{\delta}_k}{1+ C_{\eqref{hope}}+4T_k} \text{ and} \\
	\epsilon_k &< \frac{2^{-4-k}\epsilon^*}{C_{\eqref{1507}} C_{\eqref{hope}}+\varphi_X(\mathcal{L}^2(\Omega_k\setminus \Omega_{k-1}))+ (1+C_{\eqref{hope}}+4T_k^2)\mathcal{L}^2(\Omega_k\setminus \Omega_{k-1})}, 
	\end{aligned}
	\end{equation}
	where $\varphi_X(\cdot)$ is the fundamental function of $X$. In each $\Omega_k\setminus\Omega_{k-1}$, Lemma~\ref{Segregation} gives a grid of squares $Q_i^k$ which cover most (up to a set of measure $\delta_k$) of $\Omega_{k}\setminus \Omega_{k-1}$. Now we focuss on calculations for a fixed $k$. Finally at the end of the proof we sum over $k$.
	
	We apply Lemma~\ref{GridLock} to slightly alter the squares $Q_i^k$ and call the resulting quadrilaterals $\Q_i^k$. We deal with the set $S_k$ (the set from Lemma~\ref{Click}) later. Recall $\mathcal{L}^2(S_k)<\delta_k$.

	By Lemma~\ref{Segregation} point $vi)$ we get that if $i \notin B_k$ then also $a_i \notin A_k$. Then for all $i\notin B_k$
	$$
	(Df(a_i) = 0 \text{ or } T^{-1}_k< |Df(a_i)|< T_k) \quad \text{and} \quad (J_f(a_i) = 0 \text{ or } J_f(a_i) > 4 T_k^{-1}).
	$$ 
	On the other hand by Lemma~\ref{Segregation} point $vi)$
	\begin{equation}\label{MinnieMouse}
	\mathcal{L}^2(\bigcup_{i\in B_k}\Q_{i}^k) \leq \sum_{i\in B_k}\mathcal{L}^2(\Q_{i}^k) \leq \sum_{i\in B_k}\mathcal{L}^2(2 Q_{i}^k) \leq 4 \mathcal{L}^2(\bigcup_{i\in B_k}Q_{i}^k) \leq 4\delta_k
	\end{equation}
	because $\Q_i^k$ and $Q_{i}^k$ are pair-wise essentially disjoint.
	
	On all the quadrilaterals $\Q_i^k$ we have the estimates \eqref{1507} and if $i\notin B_k$ also the estimate \eqref{15071}. We make the following categorisation of the quadrilaterals $\Q_i^k$. We put
	$$
	\begin{aligned}
	\G_k &= \big\{\Q_i^k; 1\leq i\leq K_k, i\notin B_k, T_k^{-1}<|Df(a_i)|<T_k, J_f(a_i) >4T_k^{-1} \big\}\\
	\N_k&=\big\{\Q_i^k; 1\leq i\leq K_k, i\notin B_k, T_k^{-1}<|Df(a_i)|<T_k, J_f(a_i) = 0 \big\}\\
	\Z_k&=\big\{\Q_i^k; 1\leq i\leq K_k, i\notin B_k, |Df(a_i)|=0 \big\}\\
	\B_k &= \big\{\Q_i^k; 1\leq i\leq K_k, i\in B_k\big\}.
	\end{aligned}
	$$
	Then every quadrilateral $\Q_i^k$, $1\leq i\leq K_k$ belongs exactly to one of $\G_k$, $\N_k$, $\Z_k$ or $\B_k$. As calculated in \eqref{MinnieMouse}
	\begin{equation}\label{PrettyGood}
	\mathcal{L}^2\Big(\bigcup_{\Q\in\B_k}\Q\Big) < 4\delta_k.
	\end{equation}

	Strictly speaking to define $\hat{f}$ on $\Omega_k \setminus \Omega_{k-1}$ we need to use Lemma~\ref{Click} on both $S_k$ (to define $\hat{f}$ on $S_k \cap \Omega_k$) and on $S_{k-1}$ (to define $\hat{f}$ on $S_{k-1} \setminus \Omega_{k-1}$). To make the following easier to read we redefine each $\Omega_k$ as $\Omega_k \cup S_k$ and then to define $\hat{f}$ on $\Omega_k\setminus \Omega_{k-1}$ we only have to apply Lemma~\ref{Click} on $S_k$. 
	
	By the choice of the grid (i.e. all the squares in the grid in $\Omega_k \setminus \Omega_{k-1}$ have the same size) then each square in $\Omega_k \setminus \Omega_{k-1}$ has at most $8$ neighbours (where a neighbour is a square that shares at least one vertex). The exception is when the square is at the edge of our grid in $\Omega_k\setminus \Omega_{k-1}$ because these squares may have neighbours of half their side length but the number of neighbours is still bounded by $12$. Because, for any $x \in (Q_i^k)^{\circ}$ and any $Q_j^k \neq Q_i^k$, the only way for $x \in 2Q_j^k$ is if $Q_j^k$ is a neighbour of $Q_i^k$. Therefore
	\begin{equation}\label{FiniteOverlaps}
	\sum_{i=1}^{K_k}\chi_{2Q_i^k}(x)\leq 13
	\end{equation}
	for almost every $x \in \bigcup_i Q_i^k$.
	
	\step{2}{Defining a piece-wise linear map on each $\partial \Q_{i}^{ k}$}{AC}
	We define an injective piece-wise linear function $\hat{f}$ using Lemma~\ref{PWLF} for each $\Q_i^k$, $1\leq i \leq K_k$ we put $\hat{f}(x) = f(x)$ at each vertex $x$ of $\Q_i^k$. Especially we note that $\hat{f}$ is linear on each side of each $\Q_i^k \in \G_k$ by point $v)$ of Lemma~\ref{PWLF}.
	
	\step{3}{Defining $\hat{f}$, the piece-wise affine approximation of $f$}{AD}
	In each case $\Q \in \G_k$, $\Q \in \N_k$, $\Q \in \Z_k$ or $\Q \in \B_k$ we define an injective piece-wise affine extension of $\hat{f}$ from $\partial\Q$. The quadrilateral $\Q$ is the union of 2 triangles (divided by the SW-NE diagonal) and, as was proved in Lemma~\ref{PWLF} point $v)$, for $\Q\in \G_k$ we define $\hat{f}$ as linear on each side of each of these triangles and this definition is injective. In that case $\hat{f}$ extends as an affine map onto each triangle.
	
	If $\Q \in \B_k$ or if $\Q \in \Z_k$ then we apply a 2-piece-wise affine 2-bi-Lipschitz mapping $\Psi$, which maps $\Q$ onto $Q(0,2^{-m_k})$ and there we use the extension Theorem~\ref{extension}, which gives us $g$. We define $\hat{f} = g\circ\Psi$ on $\Q$. On $\Q\in \N_k$ we define $\hat{f}$ using Theorem~\ref{extension2}.
	
	Notice that in each case we have defined $\hat{f}$ (see the use of Lemma~\ref{PWLF} in step~3) on each side of $\partial{\Q}$  so that
	\begin{equation}
	|D_{\tau}\hat{f}| \leq C|Df(a_i)| \text{ is constant on each side of } \partial\Q.
	\end{equation}
	The last definition that needs to be made in $\Omega_k \setminus \Omega_{k-1}$ is the application of Lemma~\ref{Click} to get $\hat{f}$ on $S_k$.
	
	\step{4}{Uniform convergence estimates}{AAA}
	It suffices to combine the estimates from Lemma~\ref{Segregation} point $vii)$ and point $i)$ from Lemma~\ref{PWLF} to get that $\|\hat{f}-f\|_{L^{\infty}(\bigcup_i Q_i^k)}< 5\epsilon_k$. The `squares' $\Q_i^k$, $i=K_k+1, \dots, N_k$ are even smaller than the squares $\Q_i^k$ for $i=1 \dots K_k$ and so we may assume that the oscillation there has the same bound.

	\step{5}{Estimating the distance of $D\hat{f}$ from $Df$ in $X$}{AE}
	
	In the following we refer to the `centre' of the quadrilateral $\Q$ as $a_{\Q}$. That is if $\Q = \Q_{i}^k$ then $a_{\Q} = a_i^k$ the centre of $Q^k_i$. We estimate
	\begin{equation}\label{AllTheTerms}
	\begin{aligned}
	\|(Df - &D\hat{f}) \chi_{\Omega_k \setminus \Omega_{k-1}}\|_{X(\Omega)}
	\leq \| (Df - D\hat{f})(\sum_{i=1}^{K_k}\chi_{\Q_i^k} + \chi_{S_k})\|_{X(\Omega)}\\
	&\leq \|(\sum_{\Q \in \B_k} \chi_{\Q} + \chi_{S_k}) Df\|_{X(\Omega)} + \|\sum_{\Q \in \B_k} \chi_{\Q}  D\hat{f}\|_{X(\Omega)}\\
	&\quad +\| \chi_{S_k} D\hat{f}\|_{X(\Omega)}\\
	& \quad + \|\sum_{\Q \in \G_k} \chi_{\Q}(  Df - Df(a_{\Q}))\|_{X(\Omega)} +  \|\sum_{\Q \in \G_k} \chi_{\Q} ( D\hat{f} - Df(a_{\Q}))\|_{X(\Omega)}\\
	& \quad + \|\sum_{\Q \in \N_k} \chi_{\Q}  (Df - Df(a_{\Q}))\|_{X(\Omega)} +  \|\sum_{\Q \in \N_k} \chi_{\Q} ( D\hat{f} - Df(a_{\Q}))\|_{X(\Omega)}\\
	& \quad + \|\sum_{\Q \in \Z_k} \chi_{\Q}  Df\|_{X(\Omega)} +  \|\sum_{\Q \in \Z_k} \chi_{\Q}  D\hat{f}\|_{X(\Omega)}.
	\end{aligned}	
	\end{equation}
	By \eqref{ChooseDelta}, \eqref{PrettyGood} and Lemma~\ref{Segregation} point $ii)$ ($\mathcal{L}^2(S_k)<\delta_k$) we have that 
	\begin{equation}\label{FBad}
	\|(\sum_{\Q \in \B_k} \chi_{\Q} + \chi_{S_k}) Df\|_{X(\Omega)} < 2^{-k}\epsilon^*.
	\end{equation}  
	
	The sum of the terms
	\begin{equation}\label{SticksAndStones}
	\begin{aligned}
	w_k : = {}& \|\sum_{\Q \in \G_k} \chi_{\Q}(  Df - Df(a_{\Q}))\|_{X(\Omega)} +  \|\sum_{\Q \in \Z_k} \chi_{\Q}  Df\|_{X(\Omega)}\\
	& + \|\sum_{\Q \in \N_k} \chi_{\Q}  (Df - Df(a_{\Q}))\|_{X(\Omega)} 
	\end{aligned}
	\end{equation}
	is immediately estimated by Lemma~\ref{Segregation} $vi)$, (that is \eqref{ThisOne}),  the finite overlap property \eqref{FiniteOverlaps} and \eqref{ChooseEpsilon} as follows
	\begin{equation}\label{BreakMyBones}
	\begin{aligned}
	w_k &\leq \sum_{\Q \in \G_k\cup\N_k\cup\Z_k}\|\chi_{\Q}  (Df - Df(a_{\Q}))\|_{X(\Omega)}\\
	&\leq \sum_{\Q \in \G_k\cup\N_k\cup\Z_k}\epsilon_k\mathcal{L}^2(2Q_{i}^k)\\
	& \leq 13\epsilon_k \mathcal{L}^2(\Omega_k \setminus \Omega_{k-1})\\
	&<2^{-k}\epsilon^*.
	\end{aligned}
	\end{equation}
	
	Having estimated the `$f$'-terms in \eqref{AllTheTerms}, we proceed with the `$\hat{f}$'-terms. Because every square in $S_k$ has side length at most $2^{1-m_k}$ it holds that 
	$$
	\bigcup_{i=K_k+1}^{N_k} 2Q_i^k \subset S_k + Q(0,2^{-m_k}).
	$$ 
	By Lemma~\ref{Segregation} point $v)$ we have that $$
	\partial\Omega_k + Q(0,6\cdot 2^{-m_k}) \subset S_k \subset \partial\Omega_k + Q(0,8\cdot 2^{-m_k}).
	$$
	Therefore the tube around $\partial\Omega_k$ which has double the measure of $S_k$ is a superset of $\partial\Omega_k + Q(0,12\cdot 2^{-m_k})$. Therefore the inclusion 
	$$
	\bigcup_{i=K_k+1}^{N_k} 2Q_i^k \subset \partial\Omega_k + Q(0,9\cdot 2^{-m_k}) \subset \partial\Omega_k + Q(0,12\cdot 2^{-m_k})
	$$ 
	implies that
	$$
	\mathcal{L}^2\Big(\bigcup_{i=K_k+1}^{N_k} 2Q_i^k\Big) \leq \mathcal{L}^2\big(\partial\Omega_k + Q(0,9\cdot 2^{-m_k})\big) \leq 2\mathcal{L}^2(S_k) < 2\delta_k.
	$$
	Therefore defining $\hat{f} = g$ on $S_k$ by Lemma~\ref{Click} we get, using \eqref{Halelujah}, \eqref{ChooseDelta} and the previous estimate, the following
	\begin{equation}\label{HatS}
	\|D\hat{f}\chi_{S_k}\|_{X(\Omega)} \leq C \|Df\chi_{\bigcup_{i=K_k+1}^{N_k} 2Q_i^k}\|_{X(\Omega)} < 2^{-k} C \epsilon^*.
	\end{equation}

	Now we deal with $\| D\hat{f} \chi_{\bigcup_{\B_k}\Q_i^k}\|_{X(\Omega)}$. For all $\Q_i^k$ it holds that $\Q_i^k \subset Q(a_i,\tfrac{5}{4}2^{-m_k})$ and therefore, for each $y \in \Q_i^k$ and each $0<r<2^{-2-m_k}$ we get $Q(y,r) \subset 2Q_i^k$. Recall that each $\Q_l^k$ satisfies $\tfrac{5}{4}Q_l^k \supset \Q_l^k \supset \tfrac{3}{4}Q_l^k$ and that $2Q_i^k$ is the square intersecting the centres of all its neighbours $Q_l^k$. Therefore any `square' $\Q_l^k$ intersecting $2Q_i^k$ must be a neighbour of $\Q_i^k$ and further, because $Q_i^k$ and all $Q_l^k$ neighbouring squares have the same side length i.e. $|a_i - a_l| = \diam_{\infty}Q_i^k = \diam_{\infty}Q_l^k$ we have
	\begin{equation}\label{contained2}
		2\Q_l^k \subset \tfrac{5}{2}Q_l^k \subset (\tfrac{5}{2} + |a_i - a_l|)Q_i^k \subset 5Q_i^k.
	\end{equation}
	Call $A_{k,y,r} = \{1 \leq l \leq N_k;\Q_l^k \cap Q(y,r) \neq \emptyset\}$, then similar to \eqref{Estove} we get
	\begin{equation}\label{Estove2}
	\oint_{Q(y,r)}|D\hat{f} | \dx \mathcal{L}^2 \leq \max_{l \in A_{k,y,r}}\oint_{\Q_l^k\cap Q(y,r)}|D\hat{f} | \dx \mathcal{L}^2 \leq \max_{l \in A_{k,y,r}} \| D\hat{f} \|_{L^\infty(\Q_l^k)}.
	\end{equation}
	The combination of \eqref{SpiderMan}, \eqref{SuperMan} and \eqref{1507} gives that for each $\Q_l^k \in \B_k$ and $\Q_l^k \in \Z_k$
	\begin{equation}\label{Nutter}
		\| D\hat{f} \|_{L^\infty(\Q_l^k)} \leq C \max_{l \in A_{k,y,r}} \oint_{2Q_l^k}|Df|\dx \mathcal{L}^2
	\end{equation}
	and the details are already in step~\ref{D} of the proof of Lemma~\ref{Click}. Considering \eqref{CoolAssEquation} \eqref{ThisOne} and \eqref{1507} we easily see that the above holds also for $\Q_l^k \in \N_k$. For $\Q_l^k \in \G_k$ it is immediate from \eqref{ThisOne}. We continue \eqref{Estove2} using \eqref{Nutter} and \eqref{contained2} to get
	\begin{equation}\label{Party1}
	\begin{aligned}
	\oint_{Q(y,r)}|D\hat{f}| \dx \mathcal{L}^2 &\leq \max_{l \in A_{k,y,r}} \| D\hat{f} \|_{L^\infty(\Q_l^k)}\\
	& \leq \max_{l \in A_{k,y,r}} C\,\oint_{2Q_l^k}|Df| \dx \mathcal{L}^2\\
	&\leq C\oint_{5Q_i^k}|Df| \dx \mathcal{L}^2 \leq C M(|Df|)(y).
	\end{aligned}
	\end{equation}

	Now we consider a $y\in \Q_i^k$ and $r>2^{-2-m_k}$. Using \eqref{Nutter} we have
	$$
	\begin{aligned}
	\int_{Q(y,r)}|D\hat{f}| \dx \mathcal{L}^2 &= \sum_{l\in A_{k,y,r}}\int_{Q(y,r)\cap\Q_l^k}|D\hat{f}| \dx \mathcal{L}^2\\
	&\leq \sum_{l \in A_{k,y,r}}\mathcal{L}^2(Q(y,r)\cap\Q_l^k)\|D\hat{f}\|_{\infty}\\
	&\leq C\sum_{l \in A_{k,y,r}}\mathcal{L}^2(Q(y,r)\cap\Q_l^k)\oint_{2Q_i^k}|Df| \dx \mathcal{L}^2\\
	&\leq C\sum_{l \in A_{k,y,r}}\int_{2Q_l^k}|Df| \dx \mathcal{L}^2.
	\end{aligned}
	$$
	Now we use the fact that $8r>2^{1-m_k}$ and $\Q_l^k \subset 2Q_l^k$ to see that if $l\in A_{k,y,r}$ then $Q(y,r)$ intersects $2Q_l^k$ and so $Q(y,9r) \supset 2Q_l^k \supset \Q_l^k$. Therefore, using \eqref{FiniteOverlaps} we get
	$$
	\sum_{l \in A_{k,y,r}}\int_{2Q_l^k}|Df| \dx \mathcal{L}^2
	\leq 13 \int_{Q(y,9r)}|Df| \dx \mathcal{L}^2
	$$
	for $r>2^{-2-m_k}$. Thus
	\begin{equation}\label{Party2}
	\begin{aligned}
	\oint_{Q(y,r)}|D\hat{f}| \dx \mathcal{L}^2
	&\leq C\oint_{Q(y,9r)}|Df| \dx \mathcal{L}^2\\
	&\leq C M(|Df|)(y).
	\end{aligned}
	\end{equation}
	By taking supremum over all $0 < r < \infty$ in \eqref{Party1} and \eqref{Party2} on the left-hand side we obtain that
	$$
	M(|D\hat{f}|)(x)\leq C M(|Df|)(x)
	$$
	for $x \in \Omega_k \setminus (\Omega_{k-1} \cup S_k)$.
	
	By the norm comparison \eqref{norm-comparison} and Lemma~\ref{Segregation}~$vi)$ and the $\varepsilon$-$\delta$ continuity of the norm we obtain that
	\begin{equation}\label{HatB}
	\left\|D\hat{f}\chi_{\bigcup_{\mathcal{B}_k}Q}\right\|_{X}\leq C	\left\| Df\chi_{\bigcup_{\mathcal{B}_k}Q}\right\|_{X}\leq C2^{-k}\epsilon^*.
	\end{equation}

	We estimate the term $\|\sum_{\Q \in \G_k} \chi_{\Q} ( D\hat{f} - Df(a_{\Q}))\|_{X(\Omega)}$ in \eqref{AllTheTerms} as follows. For each $\Q \in \G_k$ we have from Lemma~\ref{Segregation} point $iv)$ the estimate
	$$
	2^{1+m_k}\|f(x) - f(a_{\Q}) - Df(a_{\Q})(x-a_{\Q})\|_{\infty} < \epsilon_k
	$$
	and $\hat{f}(x) = f(x)$ at each $x$, vertex of $\Q$. Therefore on both of the triangles of $\Q$ we have $|D\hat{f} - Df(a_{\Q})| < 4\epsilon_k$. Therefore on all $\Q \in \G_k$ we have $\|D\hat{f} - Df(a_{\Q})\|_{L^{\infty}(\Q)} < 4\epsilon_k$. This means, by \eqref{ChooseEpsilon}, we can estimate
	\begin{equation}\label{HatG}
	\begin{aligned}
		\Big\|[D\hat{f} - Df(a_{\Q})]\sum_{\Q\in\G_k}\chi_{\Q}\Big\|_{X(\Omega)}
		&\leq \Big\|4\epsilon_k \sum_{\Q\in\G_k}\chi_{\Q}\Big\|_{X(\Omega)}\\
		& < 4\epsilon_k \varphi(\mathcal{L}^2(\Omega_k\setminus \Omega_{k-1}))\\
		&  < 2^{-k}\epsilon^*.
	\end{aligned}
	\end{equation}
	
	We estimate the term $\|\sum_{\Q \in \Z_k} \chi_{\Q}  D\hat{f}\|_{X(\Omega)}$ as follows. For each $\Q \in \Z_k$ we have from Corollary~\ref{RePara}, \eqref{15071} that
	$$
	\begin{aligned}
	\|D\hat{f} \|_{L^{\infty}(\Q)} 
	&< C_{\eqref{hope}} \oint_{\partial \Q} |D_{\tau} \hat{f} |\dx \mathcal{H}^1\\
	&< C_{\eqref{hope}}C_{\eqref{1507}}\epsilon_{k} \oint_{2Q} |D f |\dx \mathcal{L}^2 \\
	& <C_{\eqref{hope}}C_{\eqref{1507}}\epsilon_{k}^2 \\
	&< \epsilon_k. \\
	\end{aligned}
	$$
	Therefore using \eqref{ChooseEpsilon}
	\begin{equation}\label{HatZ}
	\|D\hat{f}\sum_{\Q\in\Z_k}\chi_{\Q}\|_{X(\Omega)}\leq \|\epsilon_k \sum_{\Q\in\Z_k}\chi_{\Q}\|_{X(\Omega)} < \epsilon_k\varphi_{X}(\mathcal{L}^2(\Omega_k\setminus \Omega_{k-1})) < 2^{-k}\epsilon^*.
	\end{equation}

	Now let us estimate the $\|\sum_{\Q \in \N_k} \chi_{\Q} ( D\hat{f} - Df(a_{\Q}))\|_{X(\Omega)}$ term. We have defined $\hat{f}$ using Theorem~\ref{extension2} on each $\Q \in \N_k$. The extension on each $\Q$ has two parts, the $W_{\Q}$ part and the $\Q \setminus W_{\Q}$ part. Now let us estimate
	$$
	\begin{aligned}
	\Big\|\sum_{\Q\in\N_k}(D\hat{f} -Df(&a_{\Q}))\chi_{\Q}\Big\|_{L^1(\bigcup_{\N_k}\Q)}\\
	&\leq \sum_{ \mathcal{Q} \in \N_k} \Big(\|D\hat{f}-Df(a_{\mathcal{Q}})\|_{L^1(W_{\Q})}+\|D\hat{f}-Df(a_{\mathcal{Q}})\|_{L^1(\mathcal{Q}\setminus W_{\Q})}\Big)\\
	&\leq \sum_{\Q\in\N_k}\Big( C\mathcal{L}^2(\Q)\varepsilon_k+\mathcal{L}^2(\Q\setminus W_{\Q})\big(\|D \hat{f}\|_{L^\infty(\mathcal{Q})}+ |Df(a_{\mathcal{Q}})|\big)  \Big)\\
	&\leq \sum_{ \Q\in\N_k} \Big((1+ C_{\eqref{hope}})\mathcal{L}^2(\Q)\varepsilon_k+\varepsilon_k\mathcal{L}^2(\mathcal{Q})\cdot4T_k\Big)\\
	&\leq (1+ C_{\eqref{hope}}+4T_k) \mathcal{L}^2(\Omega_k\setminus \Omega_{k-1})\epsilon_k\\
	&< \tilde{\delta}_k \mathcal{L}^2(\Omega_k \setminus \Omega_{k-1})
	\end{aligned}
	$$
	by \eqref{ChooseEpsilon}. Of course however we have that for all $\Q \in \N_k$ that $|Df(a_{\Q})|<T_k$ and by \eqref{CoolAssEquation} that $\|D\hat{f}\|_{L^{\infty}(\Q)} \leq (1+C_{\eqref{hope}})T_k$. Therefore
	\begin{equation}\label{LInftyBound}
	\Big\|\sum_{\Q\in\N_k}(D\hat{f} -Df(a_{\Q}))\chi_{\Q}\Big\|_{L^\infty(\bigcup_{\N_k}\Q)} \leq (2+C_{\eqref{hope}})T_k.
	\end{equation}
	But by the choice of $\tilde{\delta}_k$ from Lemma~\ref{Rozumny} (see paragraph just before \eqref{ChooseEpsilon}) we have that because $\Big\|\sum_{\Q\in\N_k}(D\hat{f} -Df(a_{\Q}))\chi_{\Q}\Big\|_{L^1(\bigcup_{\N_k}\Q)} < \tilde{\delta}_k \mathcal{L}^2(\Omega_k \setminus \Omega_{k-1})$ and \eqref{LInftyBound} that
	\begin{equation}\label{HatN}
	\Big\| \sum_{\Q\in\N_k}(D\hat{f} -Df(a_{\Q}))\chi_{\Q}\Big\|_{X(\Omega)} \leq \tilde{\epsilon}\mathcal{L}^2(\Omega_k \setminus \Omega_{k-1}) =  2^{-k}\epsilon^*.
	\end{equation}
	
	From \eqref{AllTheTerms} and summing \eqref{FBad}, \eqref{BreakMyBones} (considering \eqref{SticksAndStones}), \eqref{HatB}, \eqref{HatS}, \eqref{HatG}, \eqref{HatZ} and \eqref{HatN} we get that
	$$
	\|(Df-D\hat{f})\chi_{\Omega_k \setminus \Omega_{k-1}}\|_{X(\Omega)}< C2^{-k}\epsilon^*.
	$$
	Summing this over $k$ we get that
	$$
	\|Df-D\hat{f}\|_{X(\Omega)}\leq \sum_k \|(Df-D\hat{f})\chi_{\Omega_k \setminus \Omega_{k-1}}\|_{X(\Omega)}<C\epsilon^*\sum_k 2^{-k}  = C\epsilon^*.
	$$
	From step~\ref{AAA} and the fact that $\epsilon_k < 2^{-k}\epsilon^*$ (see \eqref{ChooseEpsilon}) we know that
	$$
	\|f-\hat{f}\|_{L^{\infty}(\Omega)}< \epsilon^*.
	$$
	
	\step {6}{Smoothing piece-wise affine maps}{AA}
	
	From the above we may assume that we have a $\hat{f}$ satisfying
	$$
	\|Df-D\hat{f} \|_{X(\Omega)}<\tfrac{\epsilon}{2}
	$$
	and
	$$
	\|f-\hat{f}\|_{L^{\infty}(\Omega)}<\tfrac{\epsilon}{2}.
	$$
	Now it suffices to apply Lemma~\ref{Approx} to prove Theorem~\ref{main}.

	\step{7}{Finite-triangulation}{OO}
	
	The finite-triangulation part of Theorem~\ref{main} follows immediately from the calculations above and \cite[Section 4.2]{C}.
	\qed
	
	\subsection*{Acknowledgements}
	The authors would like to thank to the anonymous referee for their pointed comments that helped us to improve the readability of the manuscript.


\begin{thebibliography}{00}
		
		
		\bibitem{Ball}
		\by{\name{Ball}{J.M.}}
		\book{Singularities and computation of minimizers for variational problems}
		\publ{Foundations of computational mathematics (Oxford, 1999), 1--20, London Math. Soc.
			Lecture Note Ser., 284, Cambridge Univ. Press, Cambridge, 2001}
		\endbook
		
		\bibitem{Ball2} 
		\by{\name{Ball}{J.M.}}
		\book{Progress and puzzles in Nonlinear Elasticity}
		\publ{Poly-, Quasi- and Rank-One Convexity in Applied Mechanics, Springer 2010}
		\endbook
		
		
		\bibitem{BS}
		\by{\name{Bennet}{C.} and \name{Sharpley}{R.}}
		\book{Interpolation of Operators}
		\publ{Academic press,129, 1988}
		\endbook
		
	
		
		\bibitem{C}
		\by{\name{Campbell}{D.}}
		\paper{Diffeomorphic approximation of Planar Sobolev Homeomorphisms in Orlicz-Sobolev spaces}
		\jour{J. Funct. Anal.}
		\vol{273}
		\pages{125--205}
		\yr{2017}
		\endpaper
		
		
		\bibitem{CCPS}
		\by{\name{Cavaliere} {P.}, \name{Cianchi} {A.}, \name{Pick} {L.} and \name{Slav\' ikov\' a} {L.}}
		\paper{Norms supporting the Lebesgue differentiation theorem} 
		\jour{Communications in Contemporary Mathematics}
		\vol{20}
		\pages{1-33}
		\yr{2018}
		\endpaper
		
		
		
		\bibitem{DPP}
		\by{\name{De Philippis}{G.} and \name{Pratelli}{A.}}
		\paper{The closure of planar diffeomorphisms in Sobolev spaces}
		\jour{Ann. Inst. H. Poincar\'e Anal. Non Lin\'eaire}
		\vol{37\nom 1}
		\pages{181--224}
		\yr{2020}
		\endpaper
		
		\bibitem{HK}
		\by{\name{Hencl}{S.} and \name{Koskela}{P.}}
		\book{Lectures on Mappings of finite distortion}
		\publ{Lecture Notes in Mathematics 2096, Springer, 2014, 176pp}
		\endbook
		
		\bibitem{HP}
		\by{\name{Hencl}{S.} and \name{Pratelli}{A.}}
		\paper{Diffeomorphic approximation of $W^{1,1}$ planar Sobolev homeomorphisms}
		\jour{J. Eur. Math. Soc. (JEMS)}
		\vol{20\nom 3}
		\pages{597--656}
		\yr{2018} 
		\endpaper
		
		
		\bibitem{IKO}
		\by{\name{Iwaniec}{T.}, \name{Kovalev}{L.V.} and \name{Onninen}{J.}}
		\paper{Diffeomorphic Approximation of Sobolev Homeomorphisms}
		\jour{Arch. Rational Mech. Anal.}
		\vol{201\nom 3}
		\pages{1047--1067}
		\yr{2011}
		\endpaper
		
		\bibitem{IKO2}
		\by{\name{Iwaniec}{T.}, \name{Kovalev}{L.V.} and \name{Onninen}{J.}}
		\paper{Hopf Differentials and Smoothing Sobolev Homeomorphisms }
		\jour{International Mathematics Research Notices}
		\vol{14}
		\pages{3256--3277}
		\yr{2012}
		\endpaper
		
		\bibitem{MP} 
		\by{\name{Mora-Corral}{C.} and \name{Pratelli}{A.}} 
		\paper{Approximation of piece-wise Affine Homeomorphisms by Diffeomorphisms}
		\jour{J Geom. Anal}
		\vol{24}
		\pages{1398--1424}
		\yr{2014}
		\endpaper
		
		
		
		\bibitem{BiP}
		\by {\name{Pratelli}{A.}}
		\paper{On the bi-Sobolev planar homeomorphisms and their approximation}
		\jour{Nonlinear Analysis}
		\vol{154}
		\pages{258--268}
		\yr{2017}
		\endpaper
		
		\bibitem{PR}
		\by {\name{Pratelli}{A.} and \name{Radici}{E.}}
		\paper{Approximation of planar BV homeomorphisms by diffeomorphisms}
		\jour{Journal of Functional Analysis}
		\vol{276 \nom 3}
		\pages{659--686}
		\yr{2019}
		\endpaper
		
		
		\bibitem{PR2}
		\by {\name{Pratelli}{A.} and \name{Radici}{E.}}
		\paper{On the planar minimal BV extension problems}
		\jour{Rend. Lincei Mat. Appl.}
		\vol{29 \nom 3}
		\pages{511--555}
		\yr{2018}
		\endpaper
		
		\bibitem{ER}
		\by {\name{Radici}{E.}}
		\paper{A planar Sobolev extension theorem for piece-wise linear homeomorphisms}
		\jour{Pacific Journal of Mathematics}
		\vol{239 \nom 2}
		\pages{405--418}
		\yr{2016}
		\endpaper
		
	\end{thebibliography}
\end{document}